\documentclass{amsart}
\usepackage[utf8]{inputenc}

\title{Sympletic reduction of the sub-Riemannian geodesic flow for metabelian nilpotent groups}

 \author[Bravo-Doddoli]{Alejandro Bravo-Doddoli}
  \address{A. Bravo-Doddoli;
 {University of Michigan, 530 Church St, Ann Arbor, MI 48109, United States
 \\
 \href{abravodo@umich.edu}{abravodo@umich.edu}}
 }
 \author[Le Donne]{Enrico Le Donne}
\address{E. Le Donne;
University of Fribourg, Chemin du Mus\'ee~23, 1700 Fribourg, Switzerland 
\\
\href{enrico.ledonne@unifr.ch}{enrico.ledonne@unifr.ch}}

	\author[Paddeu]{Nicola Paddeu}
 \address{N. Paddeu;
 {University of Fribourg, Chemin du Mus\'ee~23, 1700 Fribourg, Switzerland  
 \\
 \href{nicola.paddeu@unifr.ch}{nicola.paddeu@unifr.ch}}
}

\date{Sep 2023}

\usepackage[english]{babel}
\usepackage[utf8]{inputenc}
\usepackage{libertine}
\usepackage{graphicx}
\usepackage{floatflt}
\usepackage{blindtext}
\usepackage{enumitem}
\usepackage{amsthm}
\usepackage{subfig}
\usepackage{listings}
\usepackage{listingsutf8}
\usepackage{amsmath}
\usepackage{framed}
\usepackage{minibox}
\usepackage{float}
\usepackage{wrapfig}
\usepackage{longtable}
\usepackage[strict]{changepage}
\usepackage{pgfplots}
\usepackage{nicefrac}
\usepackage{units}
\usepackage{etoolbox,tikz}
\usepackage{bbm}
\usepackage{hyperref}

\usetikzlibrary{external}
\usetikzlibrary{cd}

\AtBeginEnvironment{tikzcd}{\tikzexternaldisable}
\AtEndEnvironment{tikzcd}{\tikzexternalenable}

\usepackage{amsmath}
\usepackage{amsfonts}

\usepackage{amsthm}
\usepackage {amsmath, amssymb}
\usepackage{pb-diagram}
\usepackage{tikz-cd}
\usepackage{mathtools}
\usepackage{amssymb}

\usetikzlibrary{matrix}
\pgfplotsset{width=11cm,compat=1.9}
\usepgfplotslibrary{external}
\newtheorem{Theorem}{Theorem}[section]
\newtheorem{defi}[Theorem]{Definition}
\newtheorem{Prop}[Theorem]{Proposition}
\newtheorem{Cor}[Theorem]{Corollary}
\newtheorem{lemma}[Theorem]{Lemma}
\newtheorem{Remark}[Theorem]{Remark}
\newtheorem{Conjecture}[Theorem]{Conjecture}
\theoremstyle{plain}
\numberwithin{figure}{section}

\renewcommand\labelenumi{(\roman{enumi})}
\renewcommand\theenumi\labelenumi

\DeclareMathOperator{\Span}{Span}
\DeclareMathOperator{\id}{id}

\DeclareMathOperator{\Vecc}{Vec}

\DeclareMathOperator{\Ad}{Ad}

\DeclareMathOperator{\Eng}{Eng}

\DeclareMathOperator{\Lie}{Lie}

\DeclareMathOperator{\Length}{Length}

\DeclareMathOperator{\hor}{hor}

\newcommand{\df}{\mathrm{d}}
\newcommand{\rr}{\rightarrow}

\newcommand{\sslash}{\mathbin{/\mkern-6mu/}}
\newcommand{\Symp}[3]{#1 \sslash_{#3}#2}

\begin{document}
\maketitle

\begin{abstract}
We consider nilpotent Lie groups for which the derived subgroup is abelian. We equip them with subRiemannian metrics and we study the normal Hamiltonian flow on the cotangent bundle.
We show a correspondence between normal trajectories and polynomial Hamiltonians in some euclidean space. We use the aforementioned correspondence to give a criterion for the integrability of the normal Hamiltonian flow. As an immediate consequence, we show that in Engel-type groups the flow of the normal Hamiltonian is integrable.
For Carnot groups that are semidirect products of two abelian groups, we give a set of conditions that normal trajectories must fulfill to be globally length-minimizing. 
Our results are based on a symplectic reduction procedure.
\end{abstract}

\tableofcontents

\section{Introduction}
In this paper we focus on nilpotent Lie groups, equipped with left-invariant subRiemannian metrics. Roughly speaking, we fix a left-invariant \emph{distribution}, i.e., a subbundle of the tangent bundle, and we define the distance between two points to be the infimum of absolutely continuous curves that are tangent to the distribution almost everywhere and join the two points. The length is measured with respect to some left-invariant Riemannian metric. Our aim is to study length-minimizing curves in \emph{metabelian} groups, that is, groups in which the derived subgroup is abelian.
We will focus on \emph{normal trajectories}, locally length-minimizing curves admitting a lift to the cotangent bundle of the group, called \emph{normal extremal}, solving the Hamilton equation for a smooth Hamiltonian, the \emph{normal Hamiltonian}, see \eqref{e defnormalhamiltonian}. Although not all length-minimizing curves are normal trajectories \cite{montgomery1994abnormal}, in every subRiemannian manifold there exists an open dense set whose points are connected to the origin by a length-minimizing normal trajectory \cite[Theorem 1]{agrachev2009sR_manifold_have_points_of_smoothness}. As a consequence, understanding the flow of the normal Hamiltonian is an important step in understanding the subRiemannian distance on subRiemannian nilpotent groups. The main idea of this paper is that the normal Hamiltonian associated to a left-invariant subRiemannian structure on a Lie group is invariant under a canonical action of the Lie group on its cotangent bundle. The symmetries encoded by this action are related to the presence of prime integrals of the normal Hamiltonian flow 
\cite{noether1971invariant}, and the presence of this prime integrals simplifies the study of normal trajectories. \\
We develop the idea described above using the formalism of symplectic reductions, first introduced by Marsden and Weinstein \cite{Sym-reduc}. We restrict to the action of an abelian subgroup of the Lie group on its cotangent bundle to avoid the technicalities that would arise from the action of a non-commutative group.

\begin{Theorem}\emph{(After Marsden-Weinstein, \cite[Theorem 1]{Sym-reduc})}
Let $G$ be a  
Lie group with left-invariant subRiemannian structure, let 
$A<G$ be an abelian subgroup. 
Then, the action of $A$ on $G$ by left multiplication induces a Hamiltonian action of $A$ on $T^*G$, with momentum $J:T^*G\rr \Lie(A)^*$, that fixes the normal Hamiltonian $H$. For every $\mu\in\Lie(A)^*$ the symplectic reduction $\Symp{T^*G}{A}{\mu}$ 
is a smooth manifold admitting a unique symplectic form $\omega_\mu$ such that $\Pi_\mu^*\omega_\mu= i^*\omega$, where $\Pi_\mu:J^{-1}(\mu)\rr \Symp{T^*G}{A}{\mu}$ is the canonical projection, $i:J^{-1}(\mu)\rr T^*G$ is the inclusion and $\omega$ is the standard symplectic form of $T^*G$. 
     \label{t existence_symplectic_reduction}
In particular, for every normal extremal $\lambda:[0,1]\rr T^*G$ there exists $\mu\in\Lie(A)^*$ such that 
$J\circ\lambda=\mu$ and the projection of $\lambda$ to the symplectic reduction $\Symp{T^*G}{A}{\mu}$ solves the Hamilton equation for the reduced Hamiltonian $H_\mu$. Conversely, every curve in a symplectic reduction $\Symp{T^*G}{A}{\mu}$, with $\mu\in\Lie(A)^*$, solving the Hamilton equations for $H_\mu$, 
is the projection of some normal extremal.
\label{c correspondence_normal_curves}
 \end{Theorem}    
In the setting of Theorem \ref{c correspondence_normal_curves}, since the momentum is constant along normal extremals, we will say that a normal extremal $\lambda:[0,1]\rr T^*G$ \emph{has momentum} $\mu\in\Lie(A)^*$ if 
$J\circ\lambda=\mu$.
We apply the result in Theorem \ref{c correspondence_normal_curves} to 
metabelian simply connected nilpotent Lie groups. For these groups, we give an explicit description of the reduced normal Hamiltonians that arise performing symplectic reductions in terms of a \emph{connection 1-form}, see Definition \ref{d connection-1-form}. 
This will give a natural generalization of a phenomenon that occurs in the Heisemberg group and in the Engel group. Indeed, it is known that in the first group the curvature of normal trajectories is constant, whereas in the latter group it is proportional to a linear function of the projection of the curve to the abelianization of the group. Both this conditions can be expressed saying that the projections of normal extremals to some symplectic reduction solve the Hamilton equations for the Hamiltonian of a particle in an electro-magnetic field. In this paper we show that the latter statement holds for every metabelian simply connected nilpotent Lie group. We stress that metabelian nilpotent groups constitute a wide class of examples, as for instance all step $3$ nilpotent groups and all jet spaces are metabelian and nilpotent.\\
Before stating the main result of this paper we introduce some notation. Let $G$ be a simply connected nilpotent Lie group with left-invariant distribution $\Delta$ and left-invariant Riemannian metric. Let $A$ be an abelian subgroup of $G$ with $[G,G]\subseteq A$. Then we have that $A$ is normal in $G$ (we use the notation $A\triangleleft G$). Moreover, $G/A$ admits a canonical metric, induced by a scalar product, that makes the projection $\Pi:G\to G/A$ a submetry. We denote with $\pi:T^*(G/A)\rr G/A$ the canonical projection. We write $||\cdot||_{T^*(G/A)}$ for the dual norm induced by the scalar product in $T(G/A)$ on $T^*(G/A)$, and we use $||\cdot||_A$ to denote the dual norm induced by the left-invariant scalar product on the distribution $\Delta$ on the left-invariant extension of $(\Lie(A) \cap \Delta_1)^*$ (we identify $\Lie(A)$ with $T_1A$). If $\eta\in \Omega^1(G,\mathbb{R})$ is a $1$-form, we write $\eta_A$ for the restriction of $\eta$ to the left-invariant extension of $\Lie(A)\cap \Delta_1$. For $\alpha\in \Omega^1(T^*G,\Lie(A))$, we define  $\Bar{\alpha}\in\Omega^{1}(G/A,\Lie(A))$ setting $\Bar{\alpha}_{Ag}(\df_g\Pi X):=\alpha_g(X)$ for all $g\in G$, for all $X\in \df L_g (\Lie(A)\cap \Delta_1)^{\perp}$ (remark that $\Bar{\alpha}$ is well defined since $\Lie(G)=\Lie(A)\oplus (\Lie(A)\cap \Delta_1)^{\perp}$).
 \begin{Theorem}
Let $G$ be a metabelian simply connected nilpotent Lie group, equipped with a left-invariant subRiemannian structure. Let $A\triangleleft G$ be an abelian subgroup with $[G,G]\subseteq A$. Then there exists an $A$-invariant $\Lie(A)$-valued $1$-form $\alpha\in\Omega^1(G,\Lie(A))$ for which the following holds: for all $\mu\in\Lie(A)^*$ there is a symplectomorphism $\varphi_\mu:\Symp{T^*G}{A}{\mu} \rr T^*(G/A)$ such that, if $H_\mu\in C^\infty(\Symp{T^*G}{A}{\mu})$ is the reduced normal Hamiltonian, there holds
\begin{equation}
    (H_\mu\circ \varphi_\mu^{-1})(\eta)=\frac{1}{2} ||\eta+ \langle\mu,\Bar{\alpha}\rangle(\pi(\eta))||_{T^*(G/A)}^2 + \frac{1}{2} ||\langle\mu,\alpha\rangle_A(q)||_A^2,
    \label{e formofHmu0}
\end{equation}
for all $\eta\in T^*(G/A)$ and
$q\in\Pi^{-1}(\pi(\eta))$.  In particular, there is a correspondence between normal extremals in $T^*G$ with momentum $\mu$ and lifts of curves in $T^*(G/A)$ solving the Hamilton equation for $H_\mu\circ\varphi_\mu^{-1}$.

     \label{t formofHmu}
\end{Theorem}  

The main idea in the proof of Theorem \ref{t formofHmu} is that we can trivialize the structure of $A$-principal bundle of $G$. Using this trivialization, we define a closed connection $1$-form $\alpha$. 
With this choice of $\alpha$, the shift $s_\mu:T^*G\rr T^*G$, defined by $s_\mu(\lambda):= \lambda-\langle\mu,\alpha\rangle$, for all $\lambda\in T^*G$, is a symplectomorphism. We define $\varphi_\mu$ as the composition of two maps: a map $\Bar{s}_\mu:\Symp{T^*G}{A}{\mu}\rr \Symp{T^*G}{A}{0}$ induced by $s_\mu|_{J^{-1}(\mu)}$, and a canonical symplectomorphism $\Bar{\psi}_0:\Symp{T^*G}{A}{0}\rr T^*(G/A)$. We then prove \eqref{e formofHmu0} using the explicit formulas of $s_\mu$, $\Bar{\psi}_0$ and of the normal Hamiltonian. \\ 

When we will need to perform explicit computations, it will be convenient to fix a system of coordinates. We choose particular exponential coordinates of second type $\Bar{\Phi}:\mathbb{R}^n\rr G/A$, see \eqref{e coord_quotient},
and we prove that there exist polynomial maps $\mathcal{A}_1,...,\mathcal{A}_{n},\beta_1,...,\beta_{n_1}:\mathbb{R}^{n} \rr \Lie(A)$ for which equation \eqref{e formofHmu0} rewrites as
     \begin{equation}
          \Tilde{H}_\mu(p,x)= \frac{1}{2} \sum_1^n| p_i+ \langle \mu,\mathcal{A}_i(x)\rangle|^2+ \frac{1}{2} \sum_1^{n_1}|\langle\mu ,\beta_l(x)\rangle|^2,  \forall (p,x)\in T^*\mathbb{R}^{n},
          \label{e formofHmu}
     \end{equation}
where $\Tilde{H}_\mu:=(H_\mu\circ \varphi_\mu^{-1}\circ T^*\Bar{\Phi}^{-1})$. To approach various problems related to normal trajectories we will either directly apply Theorem \ref{t formofHmu} or use the formulation in coordinates given by Equation \eqref{e formofHmu}. The first problem we address is the study of the integrability of the normal Hamiltonian flow. We prove that if the Hamiltonian flow of each reduction of the normal Hamiltonian is integrable and some smoothness condition on the prime integrals holds, then the flow of the normal Hamiltonian is integrable. 
We apply this result to prove the integrability of the normal Hamiltonian flow in nilpotent groups having an abelian subgroup of co-dimension $1$ and in groups of \emph{Engel-type}, see Section \ref{s engel_type} or \cite[Section 5]{ledonnemoisala}.
\begin{Cor}
Let $G$ be metabelian simply connected nilpotent Lie group, with left-invariant subRiemannian structure. Let $A\triangleleft G$ be an abelian subgroup containing $[G,G]$. Let $\varphi_\mu:\Symp{T^*G}{A}{\mu}\rr T^*(G/A)$ be the symplectomorphism coming from Theorem~\ref{t formofHmu}. Assume there exist smooth functions $f_1,...,f_{\dim(G)-\dim(A)}:T^*(G/A)\times \Lie(A)^*\rr \mathbb{R}$ such that, for all $\mu\in\Lie(A)^*$, the functions $f_1(\cdot,\mu),...,f_{\dim(G)-\dim(A)}(\cdot,\mu)\in C^{\infty}(T^*(G/A))$ are a set of independent prime integrals for $H_\mu\circ\varphi_\mu^{-1}$ that are in involution. Then the normal Hamiltonian flow is Arnold-Liouville integrable. As a consequence, if $\dim(A)=\dim(G)-1$ then the normal Hamiltonian flow is Arnold-Liouville integrable.
\label{c consequences}
\end{Cor}
\begin{Cor}
The normal Hamiltonian flow in Engel-type groups is Arnold-Liouville integrable. 
\label{c Eng}
\end{Cor}
Corollary \ref{c consequences} provides a general procedure to prove that the normal Hamiltonian flow is Arnold-Liouville integrable. We stress that all jet spaces $J(\mathbb{R},\mathbb{R}^n)$ of smooth functions from $\mathbb{R}$ to $\mathbb{R}^n$ are metabelian and admit an abelian subgroup of dimension 1, thus as a particular case of Corollary \ref{c consequences} we have that the normal Hamiltonian flow in $J(\mathbb{R},\mathbb{R}^n)$ is Arnold-Liouville integrable for all $n\in\mathbb{N}$. We refer to
\cite[Section 6.5]{tour} for a presentation of jet spaces.\\

Using Theorem \ref{t formofHmu} we can also study \emph{metric lines}, that is, globally length-minimizing curves. We focus on Carnot  groups (see Definition \ref{d Carnot}) 
that are semidirect products of two abelian groups. In these groups, we present a class of normal trajectories that cannot be globally length-minimizing. 
The basic idea is to look for conditions that prevent the existence of a $1$-parameter subgroup in the set of \emph{blow-downs} of the normal trajectory (we refer to Section \ref{s normal_metric_lines} for more details). In order to write a precise statement we first need to introduce some notation: if $G$ is a metabelian Carnot group, and $A\triangleleft G$ is a normal subgroup with $[G,G]<A$, we denote with $\Phi:\mathbb{R}^{n+m}\to G$ and $\Bar{\Phi}:\mathbb{R}^n\rr G/A$ the coordinates of second type for which \eqref{e formofHmu} holds. 
We define for all $\mu\in\Lie(A)^*$ the functions $V_\mu,F_{1,\mu},...,F_{n_1,\mu}\in C^\infty(\mathbb{R}^n)$ setting $V_\mu:=\sum_{l=1}^{n_1}\langle\mu,\beta_l\rangle^2$ and $F_{i,\mu}:=\sum_1^{n_1} \langle \mu,\beta_l\rangle\df\theta_i(\Phi^*\beta_l)$, for all $i\in\{1,...,n_1\}$. We use $\Pi_\mu:T^*G\rr \Symp{G}{A}{\mu}$ and $\pi:T^*\mathbb{R}^n\rr \mathbb{R}^n$ to denote the canonical projections.
\begin{Cor}
Let $G$ be a metabelian Carnot group, $A\triangleleft G$ an abelian subgroup with $[G,G]\subseteq A$. Denote with $V_1$ the first layer of $\Lie(G)$. Assume $(\Lie(A)\cap V_1)^\perp$ to be an abelian sub-algebra of $\Lie(G)$. 
Let $\lambda:\mathbb{R}\rr T^*G$ be a normal extremal with momentum $\mu\in\Lie(A)^*$. Define $x:=\pi\circ T^*\Bar{\Phi}\circ\Pi_\mu(\lambda)$, $V_\mu$ and $F_{i,\mu}$ using the notation introduced before the corollary. If one of the two following condition holds,
\begin{enumerate}
    \item there exists a compact set $\Omega\subseteq \mathbb{R}^n$ such that $x(t)\in\Omega $ for all $t\in\mathbb{R}$, and there is an open set $U$, with $\Omega\subseteq U$, such that $\frac{1}{2}$ is a regular value for $V_\mu|_U$;
    \item for some $i\in\{1,...,n_1\}$ the limits $\lim_{t\rr \infty}F_{i,\mu}(x(t))$ and $\lim_{t\rr -\infty}F_{i,\mu}(x(t))$ exist and do not coincide;
\end{enumerate}
then the normal trajectory associated to $\lambda$ is not a metric line.
\label{c metric_lines}
\end{Cor}%
The above corollary was already known for jet spaces, see \cite[Theorem B]{RM-ABD}. The technique used is based on the concept of Hill region and is in some part inspired by \cite{montgomeryhill}.
Corollary \ref{c metric_lines} restricts the family of curves one should deal with when looking for metric lines. For example, if $\mu\in\Lie(A)^*$ is such that $\frac{1}{2}$ is a regular value of $V_\mu$ and the region $V_\mu^{-1}\left(\left[0,\frac{1}{2}\right]\right)$ is compact, then all normal trajectories associated to extremals with momentum $\mu$ are not metric lines. Natural candidates metric lines are normal extremals $\lambda$ for which the curve  $\pi\circ T^*\Bar{\Phi}\circ\Pi_\mu(\lambda)$ converges to critical points of $V_\mu$ as time goes to $\pm\infty$. We conjecture that the associated normal trajectory will be a metric line if and only if condition (ii) of Corollary \ref{c metric_lines} is not satisfied. 
\begin{Conjecture}
In the setting of Corollary \ref{c metric_lines}, 
the normal trajectory associated to $\lambda$ is a metric line if and only if for all $i\in\{1,...,n_1\}$ the limits $\lim_{t \to -\infty} F_{i, \mu}(x(t))$ and $\lim_{t \to +\infty} F_{ i, \mu}(x(t))$ exist and
there holds $$\lim_{t \to -\infty} F_{i, \mu}(x(t)) = \lim_{t \to +\infty} F_{ i, \mu}(x(t)).$$
\end{Conjecture}
When the normal extremal is the lift of a periodic trajectory in some symplectic reduction we can prove a slighter stronger statement than the one in Corollary \ref{c metric_lines}, giving an explicit bound of the time when the normal trajectory stops being length-minimizing. The interested reader can find more details in Section~\ref{s cut}.

\subsection{Organization of the paper}
The first part of the second section of this paper is devoted to a brief presentation of symplectic manifolds and symplectic reductions. We then shortly define subRiemannian manifolds, nilpotent and Carnot groups, normal trajectories. The proof of Theorems \ref{t existence_symplectic_reduction} and \ref{t formofHmu} is contained in Section 3. The latter section also includes a presentation of the reduction in coordinates of second type and the description of the procedure to lift the flow of a reduction of the normal Hamiltonian. Corollary~\ref{c consequences} is proved in Section 4. Section 5 is devoted to the study of metric lines and the proof Corollary~\ref{c metric_lines}.  Section 6 contains an explicit bound of the cut-time for curves that project to periodic curves in some symplectic reduction.
Section 7 contains some examples and the proof of Corollary~\ref{c Eng}.

\subsection{Acknowledgements}
The authors thank Richard Montgomery for his precious help in fixing some details of the constructions made in the paper and for his openness in answering their questions. R. Montgomery was also the main source of inspiration of the ideas in Section \ref{s normal_metric_lines} and Appendix \ref{appendix_proof_Hill_region}. The authors would also like to thank Andrei Ardentov, Yuri Sachkov,
Felipe Monroy-Perez and Huang Gaofeng, 
for the interesting discussions and for sharing their insightful viewpoint on some of the topics addressed in the paper. A.B.D. was partially supported by the scholarship (CVU 619610) from `\emph{Consejo de Ciencia y Tecnologia}'
(CONACYT). E.L.D. and N.P. were partially supported by the Swiss National Science Foundation
(grant 200021-204501 `\emph{Regularity of sub-Riemannian geodesics and applications}')
and by the European Research Council  (ERC Starting Grant 713998 GeoMeG `\emph{Geometry of Metric Groups}').  
E.L.D was also partially supported by the Academy of Finland 
 (grant 322898
`\emph{Sub-Riemannian Geometry via Metric-geometry and Lie-group Theory}').

\section{Preliminaries}\label{sec preliminaries}
We present in this section the main definitions and results we will need along the paper. We briefly describe the symplectic structure of cotangent bundles and recall some results on symplectic reductions. We give the definitions of subRiemannian manifolds and Carnot groups, we present normal extremals and the normal Hamiltonian.

\subsection{Symplectic structures}
We recall here some basic facts on symplectic manifolds and symplectic reductions. We refer to \cite{Arnold} and \cite{Found-Mech} 
for an extensive presentation.

\begin{defi}
A \emph{symplectic manifold} is a smooth manifold 
with a 
closed non-degenerate differential 2-form, called the \emph{symplectic form}.
\end{defi}
On cotangent bundles, we define a canonical symplectic structure: if $G$ is a smooth manifold, we choose as the symplectic form on $T^*G$ the differential of the tautological form $s:T(T^*G)\rr \mathbb{R}$ defined by
\begin{equation*}
    s_\lambda(X):=\lambda(d_\lambda\pi X), \forall \lambda\in T^*G, \forall X\in T_\lambda T^*G,
\end{equation*}
where $\pi:T^*G\rr G$ is the canonical projection.\\
A symplectic form $\omega$ on a manifold $M$ allows us to associate to every $h\in C^{\infty}(M)$ the \emph{Hamiltonian vector field} $\Vec{h}\in \Vecc(M)$ solving
\begin{equation*}
    \omega(\Vec{h},\cdot)=\df h.
\end{equation*}
We call \emph{Hamiltonian flow} of $h\in C^{\infty}(M)$ the function $\phi_h:[0,+\infty)\times M\rr M$ solving
\begin{equation}
     \dot{\phi}_h^t(p)=\Vec{h}(\phi_h^t(p)),  \ \ \forall t\in [0,+\infty), \forall p\in M.
     \label{e def_flow}
\end{equation}
Let $I\subseteq \mathbb{R}$ be an interval. We say that a curve $\gamma:I\rr M$ \emph{solves the Hamilton equations} for $h\in C^{\infty}(M)$ if for all $t_0\in I$ we have $\gamma(t-t_0)=\phi_h^{t-t_0}(\gamma(t_0))$ for all $t\in I, t\geq t_0$.\\
When one considers two functions $f,g\in C^\infty(M)$, it is common use to describe the non-commutativity of their Hamiltonian flows in terms of their \emph{Poisson bracket}
\begin{equation}
    \{f,g\}:=\Vec{f}g.
    \label{d Poisson_bracket}
\end{equation}

The symmetries of symplectic manifolds are encoded in the action of a group via symplectomorphisms.
\begin{defi}
Let $(M,\omega)$, $(N,\omega')$ be symplectic manifolds. A diffeomorphism $\phi:M\rr N$ is a \emph{symplectomorphism} if $\phi^*\omega'=\omega$. We say that a Lie group $A$ \emph{acts via symplectomorphisms} on $M$ if there exists a smooth homomorphism between $A$ and the group of the symplectomorphisms of $M$ onto itself.
\end{defi}
The action via symplectomorphisms of a Lie group $A$ over a symplectic manifold $M$ induces an infinitesimal action 
$\sigma:\Lie(A)\to \Vecc(M)$ defined by
\begin{equation}
    \sigma(X)(p):=\frac{\df}{\df t}\exp(tX)\cdot p \big |_{t=0}, \ \ \forall p\in M.
    \label{e infinitesimal_generator}
\end{equation}

We say that the action of $A$ on $M$ is \emph{Hamiltonian} if the infinitesimal action $\sigma$ lifts to a map $J^*:\Lie(A)\rr C^\infty(M)$ that makes the following diagram commute:
\begin{equation*}
  \begin{tikzcd}
                                                          & C^{\infty}(M) \arrow[d, "\Phi"] \\
\Lie(A) \arrow[ru, "J^*", shift left] \arrow[r, "\sigma"] & \Vecc(M)           
\end{tikzcd},
\end{equation*}
where $\Phi$ is the map associating to each smooth function the corresponding Hamiltonian vector field. When the action is Hamiltonian we can define a \emph{momentum map} $J:M \rr \Lie(A)^*$, setting 
\begin{equation}
    J(p)(X):=J^*(X)(p), \forall p\in M, \forall X\in \Lie(A).
\end{equation}

Once we have a Hamiltonian action of a Lie group on a symplectic manifold we can perform a symplectic reduction.
\begin{Theorem}\emph{(Symplectic reduction, \cite[Theorem 1]{Sym-reduc})}
Let $A$ be a Lie group acting freely and properly via symplectomorphism on a symplectic manifold $(M,\omega)$. Assume the action to be Hamiltonian. Let $J:M\rr \Lie(A)^*$ be the momentum map and $\mu\in \Lie(A)^*$ be co-adjoint invariant. Then $J^{-1}(\mu)$ is an $A$-invariant smooth submanifold, and the symplectic reduction
\begin{equation*}
    \Symp{M}{A}{\mu}:= J^{-1}(\mu)/A
\end{equation*}
is a smooth manifold. Moreover, $\Symp{M}{A}{\mu}$ admits a unique symplectic form $\omega_\mu$ such that $\Pi_\mu^*\omega_\mu=i^* \omega$, with $i:J^{-1}(\mu)\hookrightarrow M$ being the canonical inclusion and $\Pi_\mu: J^{-1}(\mu)\rr\Symp{M}{A}{\mu}$ being the canonical projection.
\label{t symplecticreduction}
\end{Theorem}
\begin{Cor}\emph{(Reduction of the dynamics, \cite[Corollary 3]{Sym-reduc})}
Let $A$ be a Lie group acting freely and properly via symplectomorphism on the symplectic manifold $(M,\omega)$. Assume the action to be Hamiltonian. Let $J:M\rr \Lie(A)^*$ be the momentum map and let $h\in C^{\infty}(M)$ be $A$-invariant. The Hamiltonian flow $\phi_h:[0,+\infty)\times M\rr M$ of $h$ commutes with the action of $A$. Moreover, $J$ is constant along the trajectories of $\phi_h$. In particular, for all co-adjoint invariant $\mu\in\Lie(A)^*$, the submanifold $J^{-1}(\mu)$ is $\phi_h$ invariant. Define $h_\mu\in C^{\infty}(\Symp{M}{A}{\mu})$ setting $h|_{J^{-1}(\mu)}=h_\mu\circ \Pi_\mu$, where $\Pi_\mu:J^{-1}(\mu)\rr \Symp{M}{A}{\mu}$ is the canonical projection.
If $\phi_{h_\mu}: [0,+\infty)\times\Symp{M}{A}{\mu}\rr \Symp{M}{A}{\mu}$ is the Hamiltonian flow of $h_\mu$, then
\begin{equation*}
    \Pi_\mu \circ \phi^t_h|_{J^{-1}(\mu)}=\phi^t_{h_\mu}\circ \Pi_\mu,\quad \forall t\in [0,+\infty).
    \label{e liftequation}
\end{equation*}

\label{c dynamicreduction}
\end{Cor}
Consider a free and proper Hamiltonian action of an abelian 
Lie group $A$ on a symplectic manifold $(M,\omega)$. Fix an $A$-invariant function $h\in C^{\infty}(M)$. As a consequence of Corollary \ref{c dynamicreduction}, we have that every trajectory of the Hamiltonian flow of $h$ lies in a submanifold $J^{-1}(\mu)$, for some $\mu\in \Lie(A)^*$, and is the lift of a trajectory of the flow of $h_\mu$ in $\Symp{M}{A}{\mu}$. Consequently, if we want to study the dynamic of the Hamiltonian flow of $h$, it is sufficient to study the dynamic of the Hamiltonian flows of $h_\mu$ in $\Symp{M}{A}{\mu}$ for every $\mu\in \Lie(A)^*$.\\
When we deal with the action of a Lie group over a cotangent bundle we are able to describe the symplectic structure in the reduction using connection $1$-forms.
\begin{defi}
    Let $A$ be a Lie group acting freely and properly on an $A$-principal bundle $G$. A \emph{connection $1$-form} $\alpha\in \Omega^1(G,\Lie(A))$ is a $\Lie(A)$-valued $1$-form such that
    \begin{eqnarray}
        \alpha\circ \sigma =\id_{\Lie(A)}, \text{ where } \sigma \text{ is defined in \eqref{e infinitesimal_generator}}, \label{e condition1-1-form}\\
        \alpha(a_*v)=\Ad_a \alpha(v) \label{e condition2-1-form},
    \end{eqnarray}
    where in the left hand side of \eqref{e condition2-1-form} we interpret $a\in A$ as a function $a:G\rr G$ using the action of $A$ over $G$.
    \label{d connection-1-form}
\end{defi}
Defining a connection $1$-form allows us to explicitly describe the symplectic structure of cotangent bundle reductions. We present here a particular case of 
\cite[Theorem 6.6.3]{ortega-ratiu-momentum-maps}.
\begin{Theorem}
Let $A$ be an abelian Lie group acting freely and properly (on the left) on a manifold $G$. This action induces a Hamiltonian action on $T^*G$ with the canonical symplectic form. Fix a connection $1$-form $\alpha\in\Omega^1(G,\Lie(A))$.
Then, for all $\mu\in\Lie(A)^*$, there exists a symplectomorphism $\varphi_\mu: \Symp{T^*G}{A}{\mu}\rr T^*(G/A)$, the latter space with the symplectic form $\omega_{\text{can}}-B_\mu$, where $\omega_{\text{can}}$ is the canonical symplectic form and $B_\mu$ is defined from $B_\mu:=\Pi^*\Tilde{B}_\mu$, $\pi^*\Tilde{B}_\mu:=\df \langle\mu,\alpha\rangle$, the maps $\Pi:T^*(G/A)\rr G/A$ and $\pi:G\rr G/A$ being the canonical projections. Moreover, the map $\varphi_\mu$ is the composition of two maps: the map $\Bar{s}_\mu:\Symp{T^*G}{A}{\mu}\rr\Symp{T^*G}{A}{0}$ induced by the fiber translation $s_\mu:J^{-1}(\mu)\rr J^{-1}(0)$,
\begin{equation}
    s_\mu(\lambda):=\lambda-\langle\mu,\alpha\rangle,\quad \forall \lambda\in J^{-1}(\mu),
    \label{e defshift}
\end{equation}
and the map $\Bar{\psi}_0:\Symp{T^*G}{A}{0}\rr T^*(G/A)$, induced by the map $\psi_0: J^{-1}(0)\rr T^*(G/A)$, defined from
\begin{equation}
    (\psi_0(\lambda))(\pi_*X):=\lambda(X), \quad \forall \lambda\in J^{-1}(0), \quad \forall X\in\Vecc(G).
    \label{e defpsi0}
\end{equation}
\label{t reduction_with_connection_1-form}
\end{Theorem}
A consequence of the definition of $B_\mu$ in Theorem \ref{t reduction_with_connection_1-form} is the following:
\begin{lemma}
In the setting of Theorem \ref{t reduction_with_connection_1-form}, if $\ker(\alpha)$ is an integrable distribution then $B_\mu=0$ and $\varphi_\mu$ is a symplectomorphism between $\Symp{T^*G}{A}{\mu}$ and $\left(T^*(G/A),\omega_\text{can}\right)$. 
\label{l 0curvature}
\end{lemma}
\begin{proof}
Fix a connection $1$-form $\alpha\in \Omega^1(G,\Lie(A))$ with an integrable distribution has a kernel. 
Define the map $\hor:TG\rr \ker(\alpha)$ setting
\begin{eqnarray*}
    \hor(Y):=Y-\sigma(\alpha(Y)), \quad \forall Y\in TG.
\end{eqnarray*}
where $\sigma$ is defined in \eqref{e infinitesimal_generator}.
To show that $B_\mu$ is $0$ we have to show
\begin{equation*}
    \df \langle\mu,\alpha\rangle(\hor(X),\hor(Y))=0, \quad\forall X,Y\in \Vecc(G),
\end{equation*}
see for example \cite[Definition 2.1.7]{Hamiltonian_reduction_by_stages_marsden_and_others} or \cite[Theorem 2.1.9, Lemma 2.1.10]{Hamiltonian_reduction_by_stages_marsden_and_others}.
Fix $X,Y\in \Vecc(G)$. We have 
\begin{eqnarray*}
    \df \langle\mu,\alpha\rangle(\hor(X),\hor(Y))=\hor(X) \langle\mu,\alpha\rangle(\hor(Y))-\hor(X) \langle\mu,\alpha\rangle(\hor(Y))+\\
    \langle\mu,\alpha\rangle([\hor(X),\hor(Y)])=0,
\end{eqnarray*}
where in the last equality we used that $\hor(X),\hor(Y)\in \ker(\alpha)$ by definition of the map $\hor$ and $[\hor(X),\hor(Y)]\in \ker(\alpha)$ being $\ker(\alpha)$ an integrable distribution.
\end{proof}

\subsubsection{Abelian subgroup acting on a Lie group}\label{sec: abeliansubgroupactingonaLiegroup}
In this paper, we will be interested in the case of an abelian subgroup of a subRiemannian Lie group acting via left-translation on the cotangent bundle of the Lie group. 
Let $G$ be a Lie group and $A<G$ be an abelian subgroup. The left-action of $A$ on $G$ induces a symplectic action of $A$ on $T^*G$ given by
\begin{equation*}
    a\cdot \lambda:=(L_{a^{-1}})^*\lambda, \quad\forall a\in A, \forall \lambda\in T^*G.
\end{equation*}
This action is Hamiltonian and the momentum function $J:G\rr \Lie(A)^*$ is given by
\begin{equation}
    J(\lambda)(X):=\lambda(X^{\dagger}), \quad\forall \lambda\in T^*G,\forall X\in \Lie(A),
    \label{e defmomentumabelian}
    \end{equation}
where $X^{\dagger}$ is the right-invariant extension of $X$ (we identify $\Lie(A)$ with $T_1A$). Being $A$ abelian, every element $\mu\in \Lie(A)^*$ is co-adjoint invariant and consequently the symplectic reduction $\Symp{T^*G}{A}{\mu}$ is a smooth manifold for every $\mu\in \Lie(A)^*$. Moreover, the group $G$ has the structure of an $A$-principal bundle. The conditions \eqref{e condition1-1-form} and \eqref{e condition2-1-form} defining connection $1$-forms rewrite as
\begin{eqnarray}
    \alpha_g(\df R_g X)&=&X, \quad \forall g\in G,\forall X\in\Lie(A)\simeq T_1A, \label{e defconnection1form-groups}\\
    L_a^*\alpha&=&\alpha, \quad\forall a\in A.
\end{eqnarray}



\subsection{SubRiemannian manifolds}
We present some definitions and facts in subRiemannian geometry that we need in the paper.
\begin{defi}
  A \emph{distribution} on a smooth manifold is a subbundle of the tangent bundle. A distribution is \emph{bracket-generating} if the Lie algebra generated by sections of the distribution contains a frame of the tangent bundle in a neighborhood of every point. A \emph{subRiemannian manifold} is a smooth manifold with a bracket-generating distribution and a Riemannian metric. 
On a SubRiemannian manifold $M$ with distribution $\Delta$ and Riemannian metric $\rho$ we consider the \emph{Carnot-Carathéodory distace}  \begin{equation*}
 \begin{split}
    d_{cc}(x,y):=\inf \bigg\{ \int_0^1\sqrt{\rho(\Dot{\gamma}(t),\Dot{\gamma}(t))}\df t \ \bigg|  \ \gamma:[0,1]\rr M \text{ absolutely continuous },\\ \gamma(0)=x, \gamma(1)=y, \Dot{\gamma}(t)\in \Delta_{\gamma(t)} \text{ for a.e. } t\in[0,1]  \bigg\} .
 \end{split}
\end{equation*}
\end{defi}
By Chow's Theorem \cite[Section 0.4]{gromov1996carnot}, we have that the Carnot-Carathéodory distance is always finite.
Throughout the paper, we will only consider Lie groups equipped with left-invariant bracket-generating distributions and left-invariant Riemannian metrics.
It is known that if a curve in a subRiemannian manifold is length-minimizing, it is either a normal or an abnormal extremal (see for example \cite[Section 4]{agrachev} for a comprehensive study of length-minimizing curves in subRiemannian manifolds). This paper will focus only on the dynamics of normal trajectories in Lie groups. 
\begin{defi}
Let $G$ be a Lie group with left-invariant bracket-generating distribution $\Delta$ and left-invariant Riemannian metric $\rho$. An absolutely continuous curve is a \emph{normal trajetory} if it admits a lift $\lambda:[0,1]\rr T^*G$, called \emph{normal extremal}, solving the Hamilton equations for $H\in C^{\infty}(T^*G)$, defined by
\begin{equation}
    H(\eta):=\max_{X\in \Delta_{\pi(\eta)}}\left(\langle \eta, X \rangle-\frac{1}{2}\rho(X,X)\right)=\sum_1^d\frac{1}{2}|\langle \eta, X_i \rangle|^2, \ \ \forall \eta\in T^*G,
    \label{e defnormalhamiltonian}
\end{equation}
where $\pi:T^*G\rr G$ is the canonical projection, and $X_1,...,X_d$ is any orthonormal frame of the distribution. We call $H$ the \emph{normal Hamiltonian}. 
\end{defi}
Normal extremals are known to be smooth and locally length-minimizing \cite[Corollary~17.4]{agrachev}. 
Along the paper, we will be interested in the study of the integrability of the normal Hamiltonian flow and we will try to bound the time when normal trajectories stop being length-minimizing. We recall here the definitions of integrability and cut-time.
\begin{defi}
  Let $(M,\omega)$ be a $2n$-dimensional symplectic manifold and let $H\in C^\infty(M)$. We say that the  Hamiltonian flow of $H$ is 
  \emph{Arnold-Liouville integrable} if for all $p\in M$ there exist smooth functions $F_1,...,F_n:M\to \mathbb{R}^n$ satisfying: 
\begin{enumerate}
\item \emph{(prime integral)} there holds $\{F_i,H\} = 0$ for all $i\in\{1,...,n\}$.
\item \emph{(independence)} The rank of the Jacobian matrix of $F$ is $n$ in a neighborhood of $p$.
\item \emph{(involution)} We have $\{F_i,F_j\} = 0$ for all $i,j\in\{1,...,n\}$. 
\end{enumerate}
\end{defi}

\begin{defi}
Let $M$ be a subRiemannian manifold. The \emph{cut-time} of a geodesic $\gamma:I\rr M$ at a time $t_0\in I$ is 
\begin{equation*}
    t_{\text{cut}}(\gamma,t_0):=\sup\{t\in I \ : \ t>t_0, \ \gamma|_{[t_0,t]} \text{ is length-minimizing} \}.
\end{equation*}
\end{defi}
When the cut-time of a normal trajectory is infinite, then (up to re-parametrization) we call the length-minimizing curve a metric line.
\begin{defi}
    Let $M$ be a subRiemannian manifold. A \emph{metric line} is a curve $\gamma: \mathbb{R}\rr M$ such that
    \begin{equation*}
        d(\gamma(a),\gamma(b))=|a-b|, \forall a,b\in\mathbb{R}.
    \end{equation*}
    \label{d metric_lines}
\end{defi}

\subsection{Nilpotent groups}
In this paper we will consider simply connected metabelian nilpotent groups, equipped with left-invariant distributions and left-invariant Riemannian metrics.
\begin{defi}Let $G$ be a group. Given a subgroup $H<G$ we denote with $[H,G]$ the subgroup generated by the set $\{hgh^{-1}g^{-1} \ : \ h\in H,g\in G\}$.
For $i\in\mathbb{N}$ define $G_i<G$ setting
\begin{equation*}
    \begin{cases}
        G_1:=G;\\
        G_i:=[G_{i-1},G] \quad \text{ if } i>1;
    \end{cases}.
\end{equation*}
A group $G$ is \emph{nilpotent} if there exists $n\in\mathbb{N}$ such that $G_n=\{1_G\}$. \\
A group $G$ is \emph{metabelian} if $[G,G]$ is abelian.
\end{defi}
The interested reader can find more details on nilpotent groups in \cite{Greanleaf-Corwin}. We refer to \cite{Robinson:1995} for an extensive presentation of metabelian groups.\\
In simply connected nilpotent Lie groups there is a convenient way to choose a base of the Lie algebra in order to make exponential coordinates of second type a diffeomorphism.
\begin{defi}
    Let $G$ be nilpotent a Lie group. An ordered basis $X_1,...,X_n$ of $\Lie(G)$ is a \emph{(weak) Malcev basis} if for all $m\in\{1,...,n\}$ the vector space $\Span\{X_1,...,X_m\}$ is a sub-algebra of $\Lie(G)$
\end{defi}
It is well known that in nilpotent groups Malcev basis exist \cite[Theorem 1.1.13]{Greanleaf-Corwin}. In metabelian nilpotent groups the construction of a Malcev basis is very simple.
\begin{lemma}
    Let $G$ be a simply connected nilpotent Lie group. 
    Let $A\triangleleft G$ be an abelian subgroup with $[G,G]<A$. Let $\{Y_1,...,Y_{m},X_1,...,X_n\}\in\Lie(G)$ be a base of $\Lie(G)$ with $Y_1,...,Y_{m}$ base of $\Lie(A)$.
    Then the ordered set of vectors $Y_1,...,Y_{m},X_n,...,X_1$ is a Malcev base of $\Lie(G)$.
 \label{l Malcev_basis}
\end{lemma}
\begin{proof}
Being $A$ abelian, for all $j<m$ we have that $\Span\{Y_1,...,Y_l\}$ is an abelian sub-algebra. Moreover, for all $i\in\{1,...,n\}$ the vector subspace $\Span\{Y_1,...,Y_m,X_n,...,X_i\}$ is a sub-algebra since $[G,G]\subseteq \Lie(A)\subseteq\Span\{Y_1,...,Y_m,X_n,...,X_i\}$.
\end{proof}
Being more careful in the choice of the base of $\Lie(A)$ in Lemma \ref{l Malcev_basis}, one could find an ordered base for which the span of the first vectors is an ideal of $\Lie(G)$. However, since this is not necessary for our purposes, we prefer to stick with the weaker notion of Malcev basis and avoid to add unnecessary technical details in the proof of Lemma \ref{l Malcev_basis}.

\subsubsection{Carnot groups}
When studying globally length-minimizing curves we will focus on Carnot groups, particular nilpotent Lie groups that are of great interest in subRiemannian geometry.
\begin{defi}
A \emph{stratification} of a Lie algebra $\mathfrak{g}$ is a direct sum decomposition $\mathfrak{g}:=V_1\oplus...\oplus V_s$, with
\begin{equation*}
    [V_1,V_j]=V_{j+1}, \quad \forall j\in{1,...,s},
\end{equation*}
where $V_{s+1}:=\{0\}$.
The sub-spaces $\{V_i\}_{i\in\{1,...,s\}}$ are called \emph{layers} of the stratification.  We call $s$ the \emph{step}. A Lie algebra equipped with a stratification is \emph{stratified}.\\
A \emph{Carnot group} is a simply connected Lie group with stratified Lie algebra.
On Carnot groups we consider the left-invariant extension of the first layer as distribution and a left-invariant Riemannian metric.
\label{d Carnot}
\end{defi}
One of the properties of Carnot groups we will need in the paper is the presence of  dilations.
\begin{defi}
Let $G$ be a Carnot group with stratified Lie algebra $\Lie(G)=V_1\oplus...\oplus V_s$. The \emph{dilation} $\delta_h:G\rr G$ of a factor $h>0$ is the Lie group automorphism defined by
\begin{equation*}
    \left(\delta_h\right)_*v:=h^i v,\quad\forall v\in V_i,\forall i\in\{1,...,s\}.
\end{equation*}
\label{d dilations}
\end{defi}


\section{Symplectic reduction with respect to an abelian subgroup}
Using the notions and results recalled in Section \ref{sec preliminaries} we prove Theorem \ref{c correspondence_normal_curves}.
\begin{proof}[Proof of Theorem \ref{t existence_symplectic_reduction}]

As we recalled in Section \ref{sec: abeliansubgroupactingonaLiegroup}, the action of $A$ on $G$ is Hamiltonian with momentum function $J:G\rr \Lie(A)^*$ given by (\ref{e defmomentumabelian}). Being $A$ abelian, every $\mu\in \Lie(A)^*$ is co-adjoint invariant. Therefore, we can apply Theorem \ref{t symplecticreduction}: $J^{-1}(\mu)$ is an $A$-invariant submanifold and $\Symp{T^*G}{A}{\mu}:=J^{-1}(\mu)/A$ is a smooth submanifold admitting a unique symplectic form that makes the projection a symplectic map.
Moreover, the normal Hamiltonian $H$ (defined in \eqref{e defnormalhamiltonian}) is $A$-invariant, thus the assumptions of Corollary \ref{c dynamicreduction} are satisfied. Therefore, $J$ is invariant under the Hamiltonian flow $\phi_H:[0,+\infty)\times T^*G\rr T^*G$ of $H$ and for all $\mu\in\Lie(A)^*$ there holds
\begin{equation}
    \Pi_\mu \circ \phi^t_H|_{J^{-1}(\mu)}=\phi^t_{H_\mu}\circ \Pi_\mu, \quad\forall t\in [0,+\infty),
    \label{e liftequation2}
\end{equation} 
where $\Pi:J^{-1}(\mu)\rr \Symp{T^*G}{A}{\mu}$ is the canonical projection and $H_\mu\in C^{\infty}(\Symp{T^*G}{A}{\mu})$ is defined setting $H|_{J^{-1}(\mu)}=H_\mu\circ\Pi_\mu$.\\
Fix a normal extremal $\lambda:\mathbb{R}\rr T^*G$. Being $\lambda$ normal there holds $\lambda(t-t_0)=\phi_H^t(\lambda(t_0))$, for all $t,t_0\in\mathbb{R}$ with $t_0<t$. Thus,
$J(\lambda(t))$ is constant, and from \eqref{e liftequation2} we have that the projection of $\lambda$ to $\Symp{T^*G}{A}{J(\gamma(0))}$ solves the Hamilton equations for $H_{J(\gamma(0))}$.\\
Vice versa, fix $\mu\in\Lie(A)^*$ and define $\eta(t):=\phi_{H_\mu}^t(\eta_0)$, with $t\in[0,+\infty]$, for some $\eta_0\in \Symp{T^*G}{A}{\mu}$.
We claim that for every $\lambda_0\in T^*G$ for which 
\begin{equation}
\begin{cases}
        \lambda_0\in J^{-1}(\mu),\\
        \Pi_\mu(\lambda_0)=\eta_0,\\
\end{cases}
\label{e normallift}
\end{equation} 
the normal extremal $\lambda(t):=\phi_H^t(\lambda_0)$ projects to $\eta$. Indeed, $J$ is constant along normal trajectories thus $\lambda(t)\in J^{-1}(\mu)$ for all $t\in [0,+\infty)$. Moreover, for all $t\in[0,+\infty)$, we have from \eqref{e liftequation2} that
\begin{equation*}
    \Pi_\mu\circ\lambda(t)=\Pi_\mu\circ\phi^t_H(\lambda_0)\stackrel{\eqref{e liftequation2}}{=}\phi^{t}_{H_\mu}(\Pi_\mu(\lambda_0))=\phi^{t}_{H_\mu}(\eta_0)=\eta(t).
\end{equation*}
\end{proof}

\subsection{Symplectic reduction for metabelian groups} 
One of the steps needed in the proof of Theorem \ref{t formofHmu} is exhibiting a connection $1$-form (see Definition \ref{d connection-1-form}) with integrable distribution as kernel. We do it in the following lemma.
\begin{lemma}
Let $G$ be a metabelian simply connected nilpotent Lie group. Let $A\triangleleft G$ be an abelian subgroup such that $[G,G]\subseteq A$. Consider the action of $A$ on $G$ via left-translations. There exists a connection $1$-form $\alpha\in\Omega^1(G,\Lie(A))$ 
such that $\ker(\alpha)$ is an integrable distribution.
\label{l integrable_connection}
\end{lemma}
\begin{proof}
Choose a vector subspace $\mathcal{H}$ such that $\Lie(A)\oplus \mathcal{H}=\Lie(G)$. Define $\Phi:\Lie(A)\times \mathcal{H}\rr G$ setting
\begin{equation}
    \Phi(X,v):=\exp(X)\exp(v), \quad\forall X\in\Lie(A),\forall v\in\mathcal{H}.
    \label{e def_phi_lemma_integrable}
\end{equation}
By Lemma \ref{l Malcev_basis} and \cite[Proposition 1.2.7]{Greanleaf-Corwin} the map $\Phi$ is a diffeomorphism. 
Define $\Tilde{\alpha}\in\Omega^1(\Lie(A)\oplus \mathcal{H},\Lie(A))$ setting
\begin{equation*}
    \Tilde{\alpha}_{(X,v)}(Y,w):=Y,\quad \forall X,Y\in\Lie(A),\forall v,w\in\mathcal{H}.
\end{equation*}
Set 
\begin{equation*}
    \alpha:=(\Phi^{-1})^*\Tilde{\alpha}\in \Omega^1(G,\Lie(A)).
\end{equation*}
We claim that $\alpha$ is a connection $1$-form for the action of $A$ over $G$ and that $\ker(\alpha)$ is an integrable distribution. Indeed, 
\begin{equation*}
    \alpha\left(\df R_{\exp(X)\exp(v)}Y\right)\stackrel{\eqref{e def_phi_lemma_integrable}}{=}\alpha\left(\df_{X,v}\Phi Y\right)=\Tilde{\alpha}_{(X,v)}(Y)=Y, \forall X,Y\in\Lie(A),\forall v\in\mathcal{H}.
\end{equation*}
Fix $Y\in\Lie(A)$. Being $(\Phi^{-1}\circ L_{\exp(Y)}\circ \Phi)(X,v)=(X+Y,v) $ for all $X\in \Lie(A)$, for all $v\in \mathcal{H}$, we have $(\Phi^{-1}\circ L_{\exp(Y)}\circ \Phi)^* \Tilde{\alpha}=\Tilde{\alpha}$ and therefore $L_{\exp(Y)}^*\alpha=\alpha$. 
We proved that $\alpha$ is $A$-invariant and satisfies \eqref{e defconnection1form-groups}, thus $\alpha$ is a connection $1$-form. The kernel of $\alpha$ is an integrable distribution since $\ker(\Tilde{\alpha})$ is and $\Phi$ is a diffeomorphism.
\end{proof}
To prove Theorem \ref{t formofHmu}, we perform the symplectic reduction on $T^*G$ using the structure of $A$-principal bundle of $G$. This will allow us to get an explicit formula for the reduced Hamiltonian.
\begin{proof}[Proof of Theorem \ref{t formofHmu}]
Once for all $\mu\in\Lie(A)^*$ we prove the existence of $\varphi_\mu$ and $\alpha$ and we show that \eqref{e formofHmu0} holds, the correspondence between normal geodesics with momentum $\mu$ and curves in $T^*(G/A)$ solving Hamilton equations for $H_\mu\circ\varphi_\mu^{-1}$ is given by Theorem~\ref{c correspondence_normal_curves}.\\
By Theorem \ref{c correspondence_normal_curves} and being $A$ abelian, we have that the action of $A$ on $T^*G$ is Hamiltonian with momentum $J:T^*G\rr \Lie(A)$ defined by \eqref{e defmomentumabelian}, and 
for all $\mu\in\Lie(A)^*$ the symplectic reduction is a smooth manifold. Fix $\mu\in\Lie(A)^*$ and use Lemma \ref{l integrable_connection} to choose a connection $1$-form $\alpha$ such that $\ker(\alpha)$ is an integrable distribution. From Theorem \ref{t reduction_with_connection_1-form} and Lemma \ref{l 0curvature} we know that there exists a symplectomorphism $\varphi_\mu:\Symp{T^*G}{A}{\mu}\rr T^*(G/A)$, the target with the canonical symplectic form. Moreover, we have $\varphi_\mu=\Bar{\psi}_0\circ \Bar{s}_\mu$ where $\Bar{\psi}_0$ and $\Bar{s}_\mu$ are defined as in the statement of Theorem \ref{t reduction_with_connection_1-form}. We have to prove that \eqref{e formofHmu0} holds. The reduced Hamiltonian $H_\mu\in C^\infty(\Symp{T^*G}{A}{\mu})$ is defined from $H_\mu\circ\Pi_\mu=H|_{J^{-1}(\mu)}$, where $\Pi_\mu:J^{-1}(\mu)\rr \Symp{T^*G}{A}{\mu}$ is the canonical projection. Therefore, we have that $H_\mu\circ\varphi_{\mu}^{-1}$ is uniquely determined by 
\begin{equation}
    (H_\mu\circ\varphi_{\mu}^{-1})\circ \varphi_\mu\circ\Pi_\mu=H|_{J^{-1}(\mu)}.
    \label{e reductionma1}
\end{equation}
Since the following diagram commutes
\begin{equation}
\begin{tikzcd}
J^{-1}(\mu) \arrow[r, "s_\mu"] \arrow[d, "\Pi_\mu"']                             & J^{-1}(0) \arrow[d, "\Pi_0"'] \arrow[rd, "\psi_0"] &          \\
\Symp{T^*G}{A}{\mu} \arrow[r, "\Bar{s}_\mu"] \arrow[rr, "\varphi_\mu"', bend right] & \Symp{T^*G}{A}{0} \arrow[r, "\Bar{\psi}_0"]           & T^*(G/A)
\end{tikzcd},
\label{e diagram_symp_reduction}
\end{equation}
equation \eqref{e reductionma1} is equivalent to 
\begin{equation}
    (H_\mu\circ\varphi_{\mu}^{-1})\circ \psi_0\circ s_\mu=H|_{J^{-1}(\mu)}.
    \label{e reductionma2}
\end{equation}
To prove \eqref{e formofHmu0}, we define $\Bar{H}\in C^{\infty}(T^*(G/A))$ setting $\Bar{H}(\eta):=\frac{1}{2}||\eta+\langle\mu,\Bar{\alpha}\rangle(\pi(\eta)) ||^2_{T^*(G/A)}+\frac{1}{2}||\langle\mu,\alpha\rangle_A(q_{\pi(\eta)})||^2_A$, for all $\eta\in T^*(G/A)$, with $q_{\pi(\eta)}\in \Pi^{-1}(\pi(\eta))$, and we show
\begin{equation}
    \Bar{H}\circ \psi_0\circ s_\mu=H|_{J^{-1}(\mu)}.
    \label{e reductionma3}
\end{equation}
By \eqref{e defshift} and \eqref{e defpsi0} we have 
\begin{equation}
    (\psi_0\circ s_\mu)(\lambda)(\pi_* X)= \langle \lambda,X\rangle-\langle \langle\mu,\alpha\rangle, X\rangle,\quad \forall X\in \Vecc(G), \forall \lambda\in J^{-1}(\mu). 
    \label{e defPmu}
\end{equation}
Fix an orthonormal base $X_1,...,X_n,Y_1,...,Y_{n_1}$ of $\Delta$, with $Y_1,...,Y_{n_1}\in \Lie(A)$. For $\lambda\in J^{-1}(\mu)$, we have
\begin{eqnarray*}
    \Bar{H}(\psi_0\circ s_\mu(\lambda))\stackrel{}{=}\frac{1}{2}\sum_1^n(\langle \psi_0\circ s_\mu(\lambda)+\langle\mu,\Bar{\alpha}\rangle, \pi_*X_i\rangle)^2+\frac{1}{2}\sum_1^{n_1} \langle \langle\mu,\alpha\rangle,Y_l\rangle^2\\
   \stackrel{\eqref{e defconnection1form-groups},\eqref{e defPmu}}{=} \frac{1}{2}\sum_1^n(\langle\lambda-\langle\mu,\alpha\rangle, X_i\rangle+\langle\langle\mu,\alpha\rangle,X_i\rangle)^2+\frac{1}{2}\sum_1^{n_1} \langle \mu,Y_l\rangle^2 \\
   \stackrel{}{=}\frac{1}{2}\sum_1^n \langle\lambda,X_i\rangle^2+\frac{1}{2}\sum_1^{n_1}\langle\lambda,Y_l\rangle^2=H(\lambda).
\end{eqnarray*}
We proved \eqref{e reductionma3}, thus, since $H_\mu\circ\varphi_\mu^{-1}$ is uniquely characterized by \eqref{e reductionma2}, we proved \eqref{e formofHmu0}.
\end{proof}
\subsection{The reduction in coordinates}
In simply connected nilpotent Lie groups, it is convenient to work with exponential coordinates since the normal Hamiltonian will have a polynomial form. 
When working with quotients, it is natural to work in exponential coordinates of the second type, as writing the projection to right cosets is particularly simple in this setting. 
We introduce some notation to simplify the definition of this kind of coordinates: we will write $\prod_n^1\exp(x_iX_i)$ for $\exp(x_nX_n)\exp(x_{n-1}X_{n-1})\cdot...\cdot \exp(x_1X_1)$ and $\prod_1^m\exp(\theta_lY_l)$ for $\exp(\theta_1Y_1)\cdot...\cdot \exp(\theta_m Y_m)$. A statement equivalent to the following lemma was stated for rank $2$ Carnot groups in \cite[Theorem 2.2]{Monti-Socionovo}.
\begin{lemma}
Let $G$ be a metabelian simply connected nilpotent Lie group and $A\triangleleft G$ an abelian subgroup containing $[G,G]$. Fix a base $\{Y_1,...,Y_m\}$ of $\Lie(A)$ and complete it to a base $\{Y_1,...,Y_m,X_1,...,X_n\}$ of $\Lie(G)$. Define exponential coordinates of second type $\Phi:\mathbb{R}^{m+n}\rr G$ setting
\begin{equation}
    \Phi(\theta_1,...,\theta_m,x_1,...,x_n):=\prod_{l=1}^m\exp(\theta_lY_l)\prod_{i=n}^1\exp(x_iX_i), 
    \label{e coordinatesofsecondtype}
\end{equation}
for all  $\theta_1,...,\theta_m,x_1,...,x_n\in \mathbb{R}$.
Then
\begin{equation}
    X_i(\Phi(\theta, x))=\frac{\partial \Phi}{\partial x_i}(x,\theta)+\df R_{\Phi(\theta, x)}\mathcal{A}_i(x),  \quad \forall \theta\in\mathbb{R}^m,\forall x\in\mathbb{R}^n,
    \label{e expressionX}
\end{equation}
and
\begin{equation}
    Y_l(\Phi(\theta, x))=\df R_{\Phi(\theta, x)}\beta_l(x), \quad\forall \theta\in\mathbb{R}^m,\forall x\in\mathbb{R}^n,
    \label{e expressionY}
\end{equation} 
where $\mathcal{A}_1,...,\mathcal{A}_n,\beta_1,...,\beta_m:\mathbb{R}^n\rr \Lie(A)$ are polynomial maps (we identify $\Lie(A)$ with $T_1A$) defined 
for all $x\in\mathbb{R}^n,\theta\in\mathbb{R}^m,i\in\{1,...,n\},l\in\{1,...,m\}$, by
\begin{eqnarray}
    &\mathcal{A}_j(x):=-\Ad_{\prod_{i=n}^{1}\exp(x_iX_i)}\left(\Ad_{\prod_{i=  1}^{j-1}\exp(-x_iX_i)}X_j-X_j\right), \label{e defAi}\\
    &\beta_l(x):=\Ad_{\prod_{i=n}^{1}\exp(x_iX_i)}Y_l. \label{e defbetal}
\end{eqnarray}

\label{l coordinates}
\end{lemma}
\begin{proof}
Since $[G,G]\subseteq A$, it is clear that the $\mathcal{A}_i$'s and the $\beta_l$'s 
take value in $\Lie(A)$. Moreover,  \eqref{e expressionY} comes trivially from \eqref{e defbetal} and the fact that $A$ is abelian.\\
We have to prove \eqref{e expressionX}.
We start computing some partial derivatives of $\Phi$. Fix $j\in\{1,...,n\}$, $\theta\in\mathbb{R}^m$ and $x\in\mathbb{R}^n$. There holds
\begin{eqnarray*}
\frac{\partial \Phi}{\partial x_j}(\theta,x)&=&\frac{\df}{\df \epsilon} L_{\prod_{l=1}^m\exp(\theta_lY_l)\prod_{i=n}^j\exp(x_iX_i)}R_{\prod_{i=j-1}^{1}\exp(x_iX_i)}\exp(\epsilon X_j)\big|_{\epsilon=0}\\
&=&\frac{\df}{\df \epsilon} L_{\Phi(\theta, x)}C_{\prod_{i=  1}^{j-1}\exp(-x_iX_i)}\exp(\epsilon X_j)\big|_{\epsilon=0}\\
&=&\df L_{\Phi(\theta, x)}\Ad_{\prod_{i=  1}^{j-1}\exp(-x_iX_i)}X_j.
\end{eqnarray*}
Consequently,
\begin{eqnarray}
     \frac{\partial \Phi}{\partial x_j}(\theta,x)-X_j(\Phi(\theta, x))= \df L_{\Phi(\theta, x)}\left(\Ad_{\prod_{i=  1}^{j-1}\exp(-x_iX_i)}X_j-X_j\right)\\
    =\df R_{\Phi(\theta, x)} \Ad_{\Phi(\theta, x)}\left(\Ad_{\prod_{i=  1}^{j-1}\exp(-x_iX_i)}X_j-X_j\right)\\
    =\df R_{\Phi(\theta, x)} \Ad_{\prod_{i=n}^{1}\exp(x_iX_i)}\left(\Ad_{\prod_{i=  1}^{j-1}\exp(-x_iX_i)}X_j-X_j\right),
    \label{e inlemmacoord1}
\end{eqnarray}
where in the last equality we used $\Ad_{\prod_{i=  1}^{j-1}\exp(-x_iX_i)}X_j-X_j\in \Lie(A)$ and that $A$ is abelian. From \eqref{e defAi} and \eqref{e inlemmacoord1}, we get \eqref{e expressionX}.
\end{proof}
We have from Lemma \ref{l Malcev_basis} and \cite[Proposition 1.2.7]{Greanleaf-Corwin} that the coordinates defined by \eqref{e coordinatesofsecondtype} are a global diffeomorphism. We rewrite Theorem \ref{t formofHmu} using exponential coordinates of second type.

\begin{Cor}
Let $G$ be a metabelian simply connected nilpotent Lie group with left-invariant subRiemannian structure. Let $A\triangleleft G$ be an abelian subgroup containing $[G,G]$. For $\mu\in \Lie(A)$, let $\varphi_\mu:\Symp{T^*G}{A}{\mu}\rr T^*(G/A)$ be the symplectomorphism coming from Theorem \ref{t formofHmu}. Then there exist polynomial maps $\mathcal{A}_1,...,\mathcal{A}_{n},\beta_1,...,\beta_{n_1}:\mathbb{R}^{n} \rr \Lie(A)$ and exponential coordinates of second type $\Bar{\Phi}:\mathbb{R}^n\rr G/A$ for which equation \eqref{e formofHmu0} rewrites as \eqref{e formofHmu}.
\label{t formofHmu2}
\end{Cor}

\begin{proof}

Fix an orthonormal base $X_1,...X_n,Y_1,...,Y_{n_1}$ of the distribution with $Y_1,...,Y_{n_1}\in\Lie(A)$. Extend $Y_1,...,Y_{n_1}$ to a base $Y_1,...,Y_{m}$ of $\Lie(A)$. Define $\Phi:\mathbb{R}^{n+m}\rr G$ as in \eqref{e coordinatesofsecondtype}. By Lemma \ref{l Malcev_basis} and \cite[Theorem 1.2.12]{Greanleaf-Corwin} we have that the map $\Bar{\Phi}:\mathbb{R}^n\rr G/A$, defined from
\begin{equation}
    \Bar{\Phi}(x_1,...,x_m):=A\prod_{i=n}^1\exp(x_iX_i),  \ \ \forall x_1,...,x_n\in \mathbb{R},
    \label{e coord_quotient}
\end{equation}
is a global diffeomorphism. By \cite[Proposition 6.3.2]{marsden_ratiu-Introduction_to_mechanics_and_symmetry}, 
$\Phi$ and $\Bar{\Phi}$ induce canonical symplectomorphisms $T^*\Phi: T^*G\rr T^*\mathbb{R}^{n+m}$ and $T^*\Bar{\Phi}:T^*(G/A)\rr T^*\mathbb{R}^n$.
Choose a connection $1$-form $\alpha$ such that
\begin{equation}
    \ker(\alpha)_{\Phi}=\Span\left\{\frac{\partial \Phi}{\partial x_i}\right\}_{i=1,...,n}.
    \label{e kerofalpha}
\end{equation}
Fix $\eta\in T^*(G/A)$ and write $(p_x,x):=T^*\Bar{\Phi}(\eta)$ . If $v_i:=\pi_*X_i$ for all $i\in\{1,...,n\}$, we have
\begin{eqnarray*}
    ||\eta+ \langle\mu,\Bar{\alpha}\rangle||_{T^*(G/A)}^2& = &\sum_1^n(\langle \eta+ \langle\mu,\Bar{\alpha}\rangle(\pi(\eta)),v_i \rangle)^2\\
    & =& \sum_1^n(p_i+\langle\langle\mu,\alpha\rangle,X_i(\Phi(\theta,x))\rangle)^2\\
    &\stackrel{\eqref{e expressionX},\eqref{e kerofalpha}}{=} &\sum_1^n(p_i+\langle\langle\mu,\alpha\rangle,\df R_{\Phi(\theta,x)}\mathcal{A}_i(x)\rangle)^2\\
    &\stackrel{\eqref{e defconnection1form-groups}}{=}&\sum_1^n(p_i+\langle\mu,\mathcal{A}_i(x)\rangle)^2,
    \label{e 1sympredmetabelian}
\end{eqnarray*}
where in the second and third equation $\theta$ is some (any) element of $\mathbb{R}^m$.
Similarly by \eqref{e expressionY} and \eqref{e defconnection1form-groups} we get
\begin{equation*}
    ||\langle\mu,\alpha\rangle_A ||^2=\sum_1^{n_1}\langle \mu,\beta_l(x)\rangle^2.
\end{equation*}
Substituting the latter two equations in \eqref{e formofHmu0} we get \eqref{e formofHmu}, concluding the proof of the corollary.
\end{proof}

\subsubsection{The computation of the lift}\label{s computation_lift}

Using exponential coordinates of the second type, we are able to write an explicit set of polynomial differential equations for the lift of curves in the symplectic reduction. Let $G$ be a metabelian simply connected nilpotent Lie group with left-invariant subRiemannian structure. Choose an Abelian subgroup $A < G$ containing $[G,G]$. Fix $\mu\in \Lie(A)^*$ and a curve $\eta:[0,1]\rr \Symp{T^*G}{A}{\mu}$ solving the Hamilton equations for the reduced normal Hamiltonian $H_\mu\in C^\infty(\Symp{T^*G}{A}{\mu})$. 
A lift $\lambda:[0,1]\rr T^*G$ is a normal extremal starting from some $\lambda_0\in T^*G$ for which \eqref{e normallift} holds.
Fix an orthonormal base $X_1,...,X_n,Y_1,...,Y_{n_1}$ of the distribution, with $Y_l\in\Lie(A)$ for all $l\in\{1,...,n_1\}$. Complete $Y_1,...,Y_{n_1}$ to a base $Y_1,...,Y_m$ of $\Lie(A)$.
Denote with $T^*\Phi:T^*G\rr T^*\mathbb{R}^{n+m}$ the symplectomorphism induced by the map $\Phi$ defined in \eqref{e coordinatesofsecondtype}. By \eqref{e defnormalhamiltonian}, \eqref{e expressionX} and \eqref{e expressionY} the normal Hamiltonian is
\begin{equation}
         (H\circ T^*\Phi^{-1})(p_x,p_\theta,x,\theta)= \frac{1}{2} \sum_1^n| p_{x_i}+ \langle p_\theta,\Phi^*\mathcal{A}_i(x)\rangle|^2+ \frac{1}{2} \sum_1^{n_1}|\langle p_\theta ,\Phi^*\beta_l(x)\rangle|^2, \label{e H_in_coordinates}
\end{equation}
for all $(p_\theta,\theta)\in T^*\mathbb{R}^m$, for all $(p_x,x)\in T^*\mathbb{R}^n$.
Set $(p_x,p_\theta,x,\theta):=T^*\Phi\circ \lambda$.
The curve $(p_x,p_\theta,x,\theta)$ solves the Hamilton equations for $H\circ T^*\Phi^{-1}$, thus $p_\theta$ is constantly equal to $\Phi^*\mu$ and $(x,p_x)$ solves the Hamilton equations for the reduced Hamiltonian $\Tilde{H}_\mu$ defined in \eqref{e formofHmu} (which is exactly what we know to be true from Corollary \ref{t formofHmu2}, since we have $(x,p_x)=T^* \Bar{\Phi}\circ \varphi_\mu \circ\Pi_\mu(\lambda)$). Moreover, using the Hamilton equations for $H\circ T^*\Phi^{-1}$ and \eqref{e H_in_coordinates}, we get a differential equation for the remaining coordinates of $\lambda$:
\begin{equation}
    \dot{\theta}_i=\sum_1^n(p_{x_i}+\langle \mu, \mathcal{A}_i(x)\rangle)\df\theta_i(\Phi^*\mathcal{A}(x)) +\sum_1^{n_1} \langle \mu,\beta_l(x)\rangle\df\theta_i(\Phi^*\beta_l(x)),
   \label{e liftequation_coordinates}
\end{equation}
where $i\in\{1,...,m\}$. We use \eqref{e liftequation_coordinates} when we need to compute the end-point of a normal trajectory.
Equation \eqref{e liftequation_coordinates} implies that if there exist $t_1,t_2\in\mathbb{R}$ such that $x(t_1)=x(t_2)$ then $\dot{\theta}(t_1)=\dot{\theta}(t_2)$. This, together with the observation that the $\partial_{\theta_l}$'s are right-invariant vector-fields, leads us to the following remark.
\begin{Remark}
    Let $\gamma:\mathbb{R}\rr G$ be a normal trajectory. Assume there exist $t_1,t_2\in\mathbb{R}$ such that for $i\in\{1,2\}$ we have $\gamma(t_i)\in A$ and $\dot{\gamma}(t_i)\in T_{\gamma(t_i)}A$. Then 
    \begin{equation*}
        \dot{\gamma}(t_2)=\df R_{\gamma(t_2)\gamma(t_1)^{-1}}\dot{\gamma}(t_1).
    \end{equation*}
    \label{r rmk_on_lift}
\end{Remark}
\subsubsection{The reduced Hamiltonian in coordinates}
When we deal with a Carnot group $G$, we will usually identify $G$ with $(\mathbb{R}^{n+m},\cdot)$, using coordinates $(\theta,x)$, with $\theta\in\mathbb{R}^m,x\in\mathbb{R}^n$. We will directly exhibit a left-invariant frame $X_1,...,X_n,Y_1,...,Y_m\in \Vecc(\mathbb{R}^{m+n})$ of the tangent bundle made of vector fields of the form
\begin{eqnarray}
   X_i(\theta,x)&=&\partial_{x_i}+\sum_{k=1}^m X_i^k(x_1,...,x_{i-1})\partial_{\theta_k}, \forall i\in \{1,...,n\}  \label{e X_i_examples},\\
   Y_l(\theta,x)&= &\sum_{k=1}^m Y_l^k(x)\partial_{\theta_k}, \forall l\in\{1,...,m\},
   \label{e Y_l_examples}
\end{eqnarray}
for all $x\in\mathbb{R}^n$, for all $\theta\in\mathbb{R}^m$, for some suitable polynomial functions $X_j^k,Y_l^k:\mathbb{R}^n\rr \mathbb{R}$. Moreover, we will ask the vector fields $X_1,...,X_n,Y_1,...,Y_{n_1}$ to form an orthonormal frame of the distribution for some $n_1\leq m$.
With this identification of $G$ with $\mathbb{R}^{m+n}$ and the choice of the base $Y_1,...,Y_m,X_1,...,X_n$, the coordinates $\Phi:\mathbb{R}^{m+n}\rr \mathbb{R}^{m+n}$ defined in Lemma \ref{l coordinates} are the identity map. Moreover, we have that $\Span(Y_1,...,Y_m)$ is an abelian sub-algebra, thus we can set $A:=\exp\left(\Span\left\{Y_1,...,Y_m\right\}\right)$ and apply Theorem \ref{t formofHmu}. 
Using that $\partial_{\theta_1},...,\partial_{\theta_m}$ are right-invariant vector-fields (it is trivial to prove this statement from \eqref{e coordinatesofsecondtype}), we have for all $x\in\mathbb{R}^n$ and for all $\theta\in\mathbb{R}^m$ that
\begin{equation}
\begin{split}
    \mathcal{A}_i(x)  & \stackrel{\eqref{e expressionX}}{=}  \df R_{(x,\theta)}^{-1}(X_i(x)-\partial_{x_i})\\
    &\stackrel{\eqref{e X_i_examples}}{=}  \sum_1^m X_i^k(x)\df R_{(x,\theta)}^{-1}\partial_{\theta_k}\\
    &=  \sum_1^m X_i^k(x)\partial_{\theta_k}\\
    &= X_i(x)-\partial_{x_i}.
\end{split}
   \label{e Aisimply}
\end{equation}
Analogously,
\begin{equation}
\begin{split}
   \beta_l(x)&\stackrel{\eqref{e expressionY}}{=}\df R_{(x,\theta)}^{-1}Y_l(x)\\
&\stackrel{\eqref{e Y_l_examples}}{=}\sum Y_l^k\df R_{(x,\theta)}^{-1}\partial_{\theta_k}\\
&=\sum Y_l^k\partial_{\theta_k}\\
&=Y_l(x).
\end{split}
    \label{e betalsimply}
\end{equation}
Being $\partial_{\theta_1},...,\partial_{\theta_l}$ a frame for the right-invariant extension of $T_1A$, we can write every right-invariant extension of a co-vector $\mu\in\Lie(A)^*\simeq T^*_1A$ as
\begin{equation*}
    \mu=\sum_1^m \mu^k\df\theta_k,
\end{equation*}
with $\mu^1,...,\mu^m\in\mathbb{R}$. 
Identify the quotient $G/A$ with $\mathbb{R}^n$ so that the projection $\Pi:T^*G\rr T^*(G/A)$ is the map $\Pi: T^*\mathbb{R}^{m}\times T^*\mathbb{R}^{n}\rr T^*\mathbb{R}^n$
\begin{equation}
    \Pi(\eta_1,\eta_2)=\eta_1, \forall \eta_1\in T^*\mathbb{R}^n,\forall \eta_2\in T^*\mathbb{R}^m. 
\end{equation}
The map $\Pi$ is a symplectic map between $T^*\mathbb{R}^{m+n}$ and $T^*\mathbb{R}^n$ with their canonical symplectic forms. 
By \eqref{e formofHmu}, \eqref{e Aisimply} and \eqref{e betalsimply} we can write the reduced Hamiltonian $H_\mu\in C^{\infty}(T^*\mathbb{R}^n)$ as
\begin{equation}
    H_\mu(p_x,x)=\frac{1}{2} \sum_1^n| p_{x_i}+ \langle \mu,X_i(x)-\partial_{x_i}\rangle|^2+ \frac{1}{2} \sum_1^{n_1}|\langle \mu ,Y_l(x)\rangle|^2,
    \label{e formofHmuexamples}
\end{equation}
for all $(p_x,x)\in T^*\mathbb{R}^n$.
Equation \eqref{e formofHmuexamples} is what we will use when dealing with explicit examples. 

\section{Consequences of the symplectic reduction}
This section is dedicated to the proof of some of the immediate consequences of the symplectic reduction we perform on metabelian nilpotent groups.

\begin{proof}[Proof of Corollary \ref{c consequences}]
Denote $n:=\dim(G)$ and $d:=\dim(A)$. Assume there exists $f_1,...,f_{n-d}\in C^\infty\left(T^*(G/A)\times \Lie(A)^*\right)$ such that for all $\mu\in\Lie(A)^*$ the functions $f_1(\cdot,\mu),...,f_{n-d}(\cdot,\mu):T^*(G/A)\rr \mathbb{R}$ are a set of independent prime integrals for $H_\mu\circ\varphi_\mu^{-1}$ that are in involution. Fix a connection $1$-form $\alpha\in\Omega^1(T^*G,\Lie(A))$ with integrable distribution as a kernel, see Lemma \ref{l integrable_connection}. 
For $\mu\in\Lie(A)^*$, define the map $P_\mu:J^{-1}(\mu)\rr T^*(G/A)$ by
\begin{equation}
    P_\mu(\lambda)(\pi_*X):=\langle\lambda, X\rangle-\langle\langle\mu,\alpha\rangle, X\rangle, \ \ \forall\lambda\in\ J^{-1}(\mu), \forall X\in\Vecc(T^*G), 
\end{equation}
where $\pi:G\rr G/A$ is the canonical projection and $J:T^*G\rr\Lie(A)$ is the momentum map.
For $i\in\{1,...,n-d\}$, define $\Tilde{f}_i\in C^\infty(T^*G)$ setting for all $\lambda\in T^*G$
\begin{equation}
    \Tilde{f}_i(\lambda):=f_i(P_{J(\lambda)}(\lambda),J(\lambda)).
\end{equation}
Fix a base $Y_1,...,Y_d$ of $\Lie(A)$ and set $J_i:=J(\cdot)(Y_i)$ for all $i\in\{1,...,d\}$. We claim that the functions $\{J_1,...,J_d,\Tilde{f}_1,...,\Tilde{f}_{n-d}\}$ are a set of independent prime integrals for the flow of the normal Hamiltonian that are in involution. By \eqref{e diagram_symp_reduction} and \eqref{e defPmu} we have that the following diagram commutes
\begin{equation*}
    \begin{tikzcd}
J^{-1}(\mu) \arrow[rd, "P_\mu"] \arrow[d, "\Pi_\mu"] &          \\
\Symp{T^*G}{A}{\mu} \arrow[r, "\varphi_\mu"]  & T^*(G/A)
\end{tikzcd}.
\end{equation*}
In particular, for every $i\in\{1,...,n-d\}$ the function $\Tilde{f}_i$ is $A$-invariant, and if $\Tilde{f}_{i,\mu}\in C^\infty(\Symp{T^*G}{A}{\mu})$ is defined by $ \Tilde{f_i}|_{J^{-1}(\mu)}=\Tilde{f}_{i,\mu}\circ\Pi_\mu$, we have
\begin{equation}
    \Tilde{f}_{i,\mu}\circ \varphi_\mu^{-1}=f_i(\cdot, \mu).
    \label{e f_i-reduced}
\end{equation}
The fact that the functions $\{J_1,...,J_d,\Tilde{f}_1,...,\Tilde{f}_{n-d}\}$ are independent is an exercise. We show that they are in involution. By \cite[Proposition 1.1]{ratiu2006involution} we have for all $\mu\in\Lie(A)^*$, for all $i,j\in\{1,...,n-d\}$, that 
\begin{equation}
    \{\Tilde{f}_i,\Tilde{f}_j\}|_{J^{-1}(\mu)}=\{\Tilde{f}_{i,\mu},\Tilde{f}_{j,\mu}\}\circ\Pi_\mu. 
    \label{e f_i-reduced-bracket}
\end{equation}
For all $i,j\in\{1,...,n-d\}$, equations \eqref{e f_i-reduced} and \eqref{e f_i-reduced-bracket} together with the assumption that $\{f_i(\cdot,\mu),f_j(\cdot,\mu)\}=0$ for all $\mu\in\Lie(A)^*$, imply that $\{\Tilde{f}_i,\Tilde{f}_j\}=0$. The fact the $\Tilde{f}_i$'s are in involution with $J_l$ for all $l\in\{1,...,d\}$ is an immediate consequence of the $A$-invariance of this functions.\\
To finish the proof one needs to show that $\{J_1,...,J_d,\Tilde{f}_1,...,\Tilde{f}_{n-d}\}$ is a set of prime integrals.
The fact that $\{H,J_l\}=0$ for all $l\in\{1,...,d\}$ follows by the $A$-invariance of $H$. 
Fix $i\in\{1,...,n-d\}$. By \cite[Proposition 1.1]{ratiu2006involution} we have for all $\mu\in\Lie(A)^*$ that
\begin{equation}
    \{\Tilde{f}_i,H\}|_{J^{-1}(\mu)}=\{\Tilde{f}_{i,\mu},H_\mu\}\circ\Pi_\mu.
    \label{e H-reduced-bracket}
\end{equation}
Being $\{H_\mu\circ\varphi_\mu^{-1},f_i(\cdot,\mu)\}=0$ for all $\mu\in\Lie(A)^*$ by assumption, equations \eqref{e f_i-reduced} and \eqref{e H-reduced-bracket} imply that
$\{H,\Tilde{f}_i\}=0$. We proved that $\{J_1,...,J_d,\Tilde{f}_1,...,\Tilde{f}_{n-d}\}$ is a set of prime integrals, thus the proof of the first part of corollary is concluded.\\
Assume now $\dim(A)=\dim(G)-1$. Then, by Theorem \ref{t formofHmu}, for all $\mu\in \Lie(A)^*$ the space $\Symp{T^*G}{A}{\mu}$ is symplectomorphic to $T^*\mathbb{R}$. In particular, for every $\mu \in \Lie(A)^*$, the flow of the reduced Hamiltonian $H_\mu$ is Arnold-Liouville integrable with prime integral $H_\mu$. Thus, by the first part of the Corollary, the flow of the normal Hamiltonian is Arnold-Liouville integrable.
%
%
%
%
%
%
%
\end{proof}

As a consequence of Theorem \ref{t formofHmu} we are able to produced several examples of the Hamiltonians that could arise when performing a symplectic reduction of the normal Hamiltonian flow. 
\begin{Cor}
Let $F_1,...,F_k:\mathbb{R}^n\rr \mathbb{R}$ be polynomial maps. Set $\phi:=\sum_1^k F_i^2$. Then, there exists a metabelian Carnot group $G$, an abelian subgroup $A< G$, a co-vector $\mu\in\Lie(A)^*$ and a symplectomorphism $f:\Symp{T^*G}{A}{\mu}\rr T^*\mathbb{R}^n$ for which
\begin{equation}
     H_\mu\circ f^{-1}=||p||^2 + \phi(x), \forall (p,x)\in T^*\mathbb{R}^n,
     \label{e allelectricpotentials}
\end{equation}
where $H_\mu\in C^{\infty}(\Symp{T^*G}{A}{\mu})$ is the reduced normal Hamiltonian and $||\cdot||$ is the norm induced  on $T^*\mathbb{R}^n$ by the euclidean norm of $\mathbb{R}^n$.
\label{c allelectricpotentials}
\end{Cor}

\begin{proof}
Denote $m:=\max_{1,...,k}\deg(F_i)$ and let $N$ be the dimension of the vector space $\mathcal{P}_{n,m}$ of polynomials from $\mathbb{R}^n$ to $\mathbb{R}$ of degree less than $m$. For every multindex $I:=(i_1,...,i_n)\in\mathbb{R}^n$ use $x^I$ to denote $x_1^{i_1}\cdot...\cdot x_n^{i_n}$. By definition of $N$ there exists a family of multindexes $\{I_j\}_{j\in\{1,...,N\}}$ such that every polynomial $F_l\in  \mathcal{P}_{n,m}$, with $l\in\{1,...,k\}$, can be written as
\begin{equation}
    F_l(x)=\sum_{j=1}^N c_{l,j}x^{I_j}, \forall x\in\mathbb{R}^n,
    \label{e coefficientsPl}
\end{equation}
for suitable coefficients $c_{l,I_1},..c_{l,I_N}\in\mathbb{R}$. 
Consider $\mathbb{R}^{kN+n}$ with coordinates $\{x_i\}_{i\in\{1,...,n\}}$ and $\{\theta_{l,j}\}_{l\in\{1,...,k\},j\in\{1,...,N\}}$.
Define $G\simeq(\mathbb{R}^{Nk+n},\cdot)$ to be the Lie group with the Lie algebra of left-invariant vector fields generated by 
\begin{eqnarray}
   X_i:= \partial_{x_i}, \text{ for } i\in\{1,...,n\},\\
   Y_l:= \sum_1^N x^{I_j}\partial_{\theta_{l,j}}, \text{ for } l\in\{1,...,k\}.
\end{eqnarray}
It is an easy exercise to show that $\mathfrak{g}:=\Lie(G)$ is a stratified Lie algebra with first layer $\Span\{\{X_i\}_{i\in \{1,...,n\}},\{Y_l\}_{l\in \{1,...,k\}}\}$ and that $\Span\{\{Y_l\}_{l\in \{1,...,k\}}\} \}\oplus[\mathfrak{g},\mathfrak{g}]$ is an abelian sub-algebra. Set $A:=\log(\Span\{\{Y_l\}_{l\in \{1,...,k\}}\} \}\oplus[\mathfrak{g},\mathfrak{g}])$ and fix on $G$ the subRiemannian metric for which $X_1,...,X_m,Y_1,...,Y_k$ is an orthonormal frame of the distribution.
Choosing
\begin{equation}
    \mu:=\sum c_{l,j}\df\theta_{l,j},
\end{equation}
we have by \eqref{e formofHmuexamples} that
\begin{equation}
    H_\mu(p_x,x)=\sum_1^n p_{x_i}^2+\sum_{l=1}^k(\sum_1^N c_{l,j}x^{I_j})^2  ,\forall (p_x,x)\in T^*\mathbb{R}^n,
\end{equation}
which by \eqref{e coefficientsPl} is exactly \eqref{e allelectricpotentials}.
\end{proof}

\section{Normal metric lines}\label{s normal_metric_lines}
In this section we study the presence of metric lines (see Definition \ref{d metric_lines}) in metabelian Carnot groups 
that are semidirect products of two abelian groups. The main idea is to study the flow of the reduced normal Hamiltonian and show that under some hypothesis normal trajectories cannot have a $1$-parameter subgroup as \emph{blow-down}, and, consequently, cannot be metric lines. 


We start this section recalling the only property of metric lines that we will need in the proof of Corollary \ref{c metric_lines}. We use $\delta_h$ to denote the dilation of a factor $h\in \mathbb{R}$, see Definition~\ref{d dilations}.
\begin{Prop}\emph(\cite[Corollary 1.6]{blowups_blowdowns})
    Let $G$ be a Carnot group. Let $\gamma:\mathbb{R}\rr G$ be a metric line. For all $h\in\mathbb{R}$ define 
    \begin{equation*}
        \gamma_{h}:=\delta_{h}\circ\gamma\circ \delta_{\frac{1}{h}}.
    \end{equation*}
    Then there exists a sequence $\{h_n\}_{n\in\mathbb{N}}$, with $\lim_{n\rr\infty}h_n=0$, for which the sequence of curves $\{\gamma_{h_n}\}_{n\in\mathbb{N}}$ converges uniformly on compact sets to a $1$-parameter subgroup parametrized by arc-length as $n\to\infty$.
    \label{l blowdownmetricline}
\end{Prop}
An immediate consequence of Proposition \ref{l blowdownmetricline} is the following: 
\begin{lemma}
Let $G$ be a subRiemannian Carnot group. Let $\gamma:\mathbb{R}\rr G$ be an absolutely continuous curve parametrized by arc-length. Fix a non-trivial vector subspace $V\subseteq G/[G,G]$ and denote with $\pi_1:G/[G,G]\rr V$ the orthogonal projection. Denote with $\pi:G\rr G/[G,G]$ the canonical projection. Define $\sigma:=\pi\circ\gamma$. If $\pi_1\circ\sigma$ is bounded and 
\begin{equation}
    \limsup_{T\rr\infty}\frac{1}{T}\int_0^T \sqrt{1-||(\pi_1\circ\sigma)'(t)||^2}\df t <1, 
    \label{e derivativethetabounded0}
\end{equation}
then $\gamma$ is not a metric line.
\label{l onecomponentbounded}
\end{lemma}
\begin{proof}
Assume by contradiction that $\gamma$ is a metric line. For $h\in\mathbb{R}\setminus\{0\}$, define the curve $\sigma_{h}:\mathbb{R}\rr G/[G,G]$ setting $\sigma_{h}(t):=h\sigma(\frac{t}{h})$ for all $t\in \mathbb{R}$.
By Proposition \ref{l blowdownmetricline} we know there exists a sequence $\{h_n\}_{n\in\mathbb{N}}$, with $\lim_{n\rr\infty}h_n=0$, such that $\gamma_{h_n}$ converges uniformly on compact sets to a $1$-parameter subgroup parametrized by arc-length. In particular, there exists $v\in G/[G,G]$, with $||v||=1$, such that the $\sigma_{h_n}$ converges uniformly on compact sets to the line $L:\mathbb{R}\rr G/[G,G]$, $L(t):=tv$ for all $t\in\mathbb{R}$. Since $\pi_1\circ\sigma$ is bounded we have $\lim_{n\rr \infty}\pi_1\circ\sigma_{h_n}(t)=0$ for all $t\in\mathbb{R}$. Consequently, if $\pi_2:G/[G,G]\rr V^\perp\subseteq G/[G,G]$ is the orthogonal projection, $\pi_2\circ\sigma_{h_n}$ converges to $L$ uniformly on compact sets and therefore, for all $T\in \mathbb{R}$, we have
\begin{equation}
    1=\frac{1}{T}\int_0^T ||v||\df t=\lim_{n\rr\infty}\frac{1}{T}\int_0^T ||(\pi_2\circ\sigma_{h_n})'(t)||\df t.
    \label{e metline1}
\end{equation}
On the other hand, being $\gamma$ parametrized by arclength we have $||\pi_2\circ \sigma'||=\sqrt{1-||\pi_1\circ \sigma'||^2}$.
Therefore, \eqref{e derivativethetabounded0} implies
\begin{equation}
    \limsup_{T\rr\infty}\frac{1}{T}\int_0^T ||(\pi_2\circ\sigma)'(t)||\df t <1.
    \label{e derivativethetabounded}
\end{equation}
Consequently,
\begin{eqnarray*}
    \lim_{n\rr\infty}\frac{1}{T}\int_0^T ||(\pi_2\circ\sigma_{h_n})'(t)||\df t=\lim_{n\rr\infty}\frac{1}{T}\int_0^T \bigg|\bigg|(\pi_2\circ\sigma)'\left(\frac{t}{h_n}\right)\bigg|\bigg|\df t\\
    =\lim_{n\rr\infty}\frac{h_n}{T}\int_0^{\frac{T}{h_n}} ||(\pi_2\circ\sigma)'(t)||\df t
    \stackrel{\eqref{e derivativethetabounded}}{<}1.
\end{eqnarray*}
The latter equation implies that \eqref{e metline1} doesn't hold, therefore we reached a contradiction. We showed that if $\pi_1\circ\sigma $ is bounded and \eqref{e derivativethetabounded} holds then $\gamma$ cannot be a metric line, thus the proof of the lemma is concluded.
\end{proof}
We state now a technical lemma that gives us a tool to verify the hypothesis in Lemma~\ref{l onecomponentbounded} using Hamilton equations for some reduced normal Hamiltonians.
\begin{lemma}
Fix a scalar product on $\mathbb{R}^n$ and call $||\cdot||$ the induced norm. Let $V\in C^\infty(\mathbb{R}^n,\mathbb{R})$ be a positive smooth function. 
Fix a compact set $\Omega\subseteq \mathbb{R}^n$ such that $\frac{1}{2}$ is a regular value of $ V|_U$ for some open set $U\subseteq \mathbb{R}^n$ for which $\Omega\subseteq U$. Define $H\in C^\infty(T^*\mathbb{R}^n)$ setting 
\begin{equation}
    H(p,x):=\sum_1^n p_i^2+V(x), \ \ \forall (p,x)\in T^*\mathbb{R}^n.
\end{equation}
Fix $x_0\in \Omega$, with $V(x_0)<\frac{1}{2}$, and choose $p_0\in \mathbb{R}^n$ such that $H(p_0,x_0)=\frac{1}{2}$. Define the curve $(p,x):[0,+\infty)\rr \mathbb{R}^n$ setting
\begin{equation}
    (p(t),x(t)):=\phi_H^t(p_0,x_0),\ \ \forall t\geq 0,
    \label{e defxhill}
\end{equation}
where we denote with $\phi_H^t$ the flow of the Hamiltonian $H$, see \eqref{e def_flow}.
Assume $x(t)\in \Omega$ for all $t\geq 0$. Then 
\begin{equation}
    \limsup_{T\rr\infty} \frac{1}{T}\int_0^T\sqrt{1-||\dot{x}(t)||^2}\df t<1.
    \label{e avaregeofxdot}
\end{equation}
\label{l hill_region}
\end{lemma}
Since the proof of Lemma \ref{l hill_region} is not helpful in understanding the statement of the lemma, nor contains any idea coming from subRiemannian geometry, we prefer to postpone it to Appendix \ref{appendix_proof_Hill_region}. We are finally ready to prove Corollary \ref{c metric_lines}. 

\begin{proof}[Proof of Corollary \ref{c metric_lines}]
Being $(\Lie(A)\cap V_1)^\perp$ an abelian sub-algebra, the functions $\mathcal{A}_i$, defined by \eqref{e defAi} for all $i=1,...,n$, are constantly $0$. Let $\lambda:\mathbb{R}\rr T^*G$ be a normal extremal with momentum $\mu\in \Lie(A)^*$. Set
\begin{equation*}
    \Tilde{H}_\mu(p,x):=\sum_1^n p_i^2+V_\mu(x), \forall (p,x)\in T^*\mathbb{R}^n. 
\end{equation*}
Write $(p_\theta, p_x,\theta,x):=T^*\Phi (\lambda)$. By Corollary \ref{t formofHmu2} we have that the couple $(p_x,x)$ solves the Hamilton equations for  $\Tilde{H}_\mu$. Assume first that condition (1) holds. We have that $x$ is bounded (since $x(t)$ is in the compact set $\Omega$ for all $t\in\mathbb{R}$) and that \eqref{e avaregeofxdot} holds by Lemma \ref{l hill_region}. Consequently, we can apply Lemma \ref{l onecomponentbounded} with $V=\Bar{\pi}(\exp((\Lie(A)\cap V_1)^\perp))$, where $\Bar{\pi}:G\rr G/[G,G]$ is the canonical projection, and conclude that the normal trajectory associated to $\lambda$ is not a metric line. \\
Assume instead that condition (2) holds. By \eqref{e liftequation_coordinates} we have $\lim_{t\rr\infty}\dot{\theta_i}(t)\neq \lim_{t\rr-\infty}\dot{\theta_i}(t)$. In particular, there cannot exist a sequence $\{h_j\}$ decreasing to $0$ such that $h_j\theta_i(\frac{\cdot}{h_j})$ converges uniformely on compact sets to a linear function. As a consequence, by Proposition~\ref{l blowdownmetricline}, the normal trajectory associated to $\lambda$ is not a metric line.
\end{proof}

\section{Cut-times}\label{s cut}
In the section we improve the result in Corollary \ref{c metric_lines}, showing that for some particular normal trajectories we can give an explicit bound on the time at which they stop being length-minimizing.
In particular, we prove that if a normal extremal is the lift of an $L$-periodic curve, with $L\in\mathbb{R}$, then under some conditions the associated normal trajectory stops to be length-minimizing before time $L$. The first key idea we need is contained in the following theorem. Recall that we denote with $t_{\text{cut}}(\gamma,0)$ the supremum of times $t$ for which $\gamma|_{[0,t]}$ is length-minimizing.
\begin{Theorem}
 Let $G$ be a subRiemannian Lie group. Let $\gamma:\mathbb{R}\rr G$ be a normal trajectory with $\gamma(0)=1_G$. Assume there exists $L>0$ such that 
 \begin{equation}
 \df R_{\gamma(L)}\dot{\gamma}(0)=\dot{\gamma}(L).   
 \label{e condition_cut_time}
 \end{equation}
 Then either $t_{\text{cut}}(\gamma,0)\leq L$ or $\gamma$ is a $1$-parameter subgroup.
 \label{t cut_time}
\end{Theorem}
\begin{proof}
We show that if $\gamma$ is not a $1$-parameter subgroup, there exists a non constantly zero Jacobi vector field $J:\mathbb{R}\rr TG$ such that $J(0)=0$ and $J(L)=0$. Then the theorem will follow from \cite[Theorem 8.61, Proposition 15.6]{agrachev}. We refer to \cite[Chapter 15]{agrachev} for an extended presentation of Jacobi vector fields.
Set 
\begin{equation}
    J(t):=\dot{\gamma}(t)-\df R_{\gamma(t)}\dot{\gamma}(0).
\end{equation}
The vector field $J$ is a Jacobi vector field since it is the difference between two Jacobi vector fields: it is well known that $\dot{\gamma}(t)$ is a Jacobi vector field, and $R_{\gamma(t)}\dot{\gamma}(0)$ is a Jacobi vector field being a Killing vector field (see 
\cite[Lemma 5.15]{agrachev}). By \eqref{e condition_cut_time} we have $J(0)=0$ and $J(L)=0$. Moreover, $J$ is constantly $0$ if and only if 
\begin{equation*}
    \dot{\gamma}(t)=\df R_{\gamma(t)}\dot{\gamma}(0), \forall t\in \mathbb{R},
\end{equation*}
thus if and only if $\gamma$ is a $1$-parameter subgroup.
\end{proof}

Using the above theorem and the discussion we had in Section \ref{s computation_lift}, we are able to study the cut-time of normal extremals that project to periodic curves.
\begin{Prop}
Let $G$ be a metabelian simply connected nilpotent Lie group with left-invariant distribution $\Delta$ and left-invariant Riemannian metric. Let $A\triangleleft G$ be an abelian subgroup with $[G,G]\subseteq A$. Let $\gamma:\mathbb{R}\rr G$ be a normal trajectory and let $\lambda:\mathbb{R}\rr T^*G$ be the corresponding normal extremal. Assume $\gamma(0)=1_G$. Denote with $\mu\in \Lie(A)^*$ the momentum of $\lambda$. Call $\Pi_\mu:T^*G\rr \Symp{T^*G}{A}{\mu}$ 
the canonical projection. If $\Pi_\mu\circ \lambda$ is $L$-periodic and either
\begin{equation}
  \lambda(0)(X)=0, \ \ \forall X\in (\Lie(A)\cap \Delta)^\perp,
  \label{e condition_velocity_tangent_to_A}
\end{equation}
or $(\Lie(A)\cap \Delta)^\perp$ is an abelian sub-algebra,
then $t_\text{cut}(\gamma,0)\leq L$.
\label{c cut_time_reduction}
\end{Prop}
\begin{proof}
Assume first that condition \eqref{e condition_velocity_tangent_to_A} holds. 
Then there holds $\dot{\gamma}(0)\in T_{\gamma(0)}A$. Since $\Pi_\mu\circ \lambda$ is $L$-periodic we have also $\gamma(L)\in A$ and $\dot{\gamma}(L)\in T_{\gamma(L)}A$.
By Remark \ref{r rmk_on_lift} we have that \eqref{e condition_cut_time} holds, thus we can apply Theorem \ref{t cut_time} and conclude the proof of the first part of the proposition.\\
Assume now that $(\Lie(A)\cap \Delta)^\perp$ is an abelian sub-algebra.
Fix a base $Y_1,...,Y_m$ of $\Lie(A)$ with $Y_1,...,Y_{n_1}$ orthonormal base of $\Delta$, and complete $Y_1,...,Y_{n_1}$ to an orthonormal base $Y_1,...,Y_{n_1},X_1,...,X_n$ of the distribution. Define exponential coordinates of second type $\Phi:\mathbb{R}^{m+n}\rr G$ as in \eqref{e coordinatesofsecondtype}. Call $\Bar{\Phi}:\mathbb{R}^n\rr G/A$ the coordinates induced on the quotient, see \eqref{e coord_quotient}. Remark that being $\Span\{X_1,...,X_n\} $ abelian by assumption, equation \eqref{e defAi} implies that the $\mathcal{A}_i$'s appearing in equation \eqref{e formofHmu} are constantly zero. Thus \eqref{e formofHmu} rewrites as
\begin{equation*}
    \Tilde{H}_\mu(p_x,x)=\sum_1^{n}p_{x_i}^2+\frac{1}{2} \sum_1^{n_1}|\langle\mu ,\beta_l(x)\rangle|^2,\forall (p_x,x)\in T^*\mathbb{R}^{n},
\end{equation*}
where the $\beta_l$'s are defined from \eqref{e defbetal}.\\
Call $(p_x,x):\mathbb{R}\rr T^*\mathbb{R}^{n}$ the function $(p_x,x):=T^*\Bar{\Phi}\circ\varphi_\mu\circ \Pi_\mu(\lambda)$. Being $\Tilde{H}_\mu(p,y)=\Tilde{H}_\mu(-p,y)$ for all $(p,y)\in T^*\mathbb{R}^n$, and since $(p_x,x)$ solves the Hamilton equations for $\Tilde{H}_\mu$, also the curve $(\Tilde{p}_x,\Tilde{x}):\mathbb{R}\rr T^*\mathbb{R}^n$, defined by $(\Tilde{p}_x,\Tilde{x})(t):=(-p_x(-t),x(-t))$ for all $t\in\mathbb{R}$, solves the Hamilton equation for $\Tilde{H}_\mu$.
Call $\Tilde{\lambda}$ the lift of $T^*\Bar{\Phi}^{-1}(\Tilde{p}_x,\Tilde{x})$ at $-\lambda(0)$ and denote $\Tilde{\gamma}:=\pi\circ \Tilde{\lambda}$, where $\pi:T^*G\rr G $ is the canonical projection. 
We show that $\gamma(L)=\Tilde{\gamma}(L)$, this will imply $t_\text{cut}(\gamma,0)\leq L$
since $\gamma\neq\Tilde{\gamma}$ and $\Length(\gamma|_{[0,L]})=\Length(\Tilde{\gamma}|_{[0,L]})$.\\
Write $(\theta,x):=\Phi^{-1}(\gamma)$ and $(\Tilde{\theta},\Tilde{x}):=\Phi^{-1}(\Tilde{\gamma})$.
By definition of $(\Tilde{p}_x,\Tilde{x})$ we have 
\begin{equation}
    \Tilde{x}(t)=x(-t),\quad\forall t\in\mathbb{R},
    \label{e reversex}
\end{equation}
Being $\Pi_\mu(\lambda)$ an $L$-periodic function we have that $x$ is $L$-periodic and consequently $\Tilde{x}(L)=x(-L)=x(L)$.
Moreover, the functions $F_i:\mathbb{R}\rr \mathbb{R}$, with $i\in\{1,...,m\}$,
\begin{equation*}
    F_i:=  \sum_1^{n_1} \langle \mu,\beta_l(x)\rangle\df\theta_i(\Phi^*\beta_l(x)),
\end{equation*}
are $L$-periodic, since they are functions of $x$, and therefore
\begin{equation}
    \int_0^L F_i(t)\df t=\int_0^LF_i(-t)\df t,\quad \forall i\in\{1,...,m\}.
    \label{e Freversed}
\end{equation}
Consequently, for all $i\in\{1,...,m\}$,
\begin{equation}
    \theta_i(L)\stackrel{\eqref{e liftequation_coordinates}}{=}\int_0^L F_i(t)\df t\stackrel{\eqref{e Freversed}}{=}\int_0^LF_i(-t)\df t\stackrel{\eqref{e liftequation_coordinates}}{=}\Tilde{\theta}_i(L).
\end{equation}
We proved the claim $\gamma(L)=\Tilde{\gamma}(L)$. Thus we concluded the proof of the proposition.

\end{proof}

\section{Examples}\label{s Examples}
We present some explicit examples of the symplectic reduction procedure described in Theorem~\ref{t formofHmu}. Throughout this section, we will give a proof of Corollary~\ref{c Eng}.

\subsection{2-abelian extensions}
We present some applications of Theorem \ref{t formofHmu} to \emph{2-abelian extensions}, i.e., to nilpotent groups containing a normal abelian subgroup of dimension $2$. The notation used in naming the following groups mainly follows \cite{Enr-fr}. We start with simple examples to help the reader to get familiar with the procedure used in Theorem \ref{t formofHmu}. 
\subsubsection{$F_{23}$ or Cartan group as $2$-abelian extension}
Let $F_{23}$ be the free-nilpotent Lie group of rank $2$ and step $3$, also known as Cartan group. 
The Lie algebra of $F_{23}$ is spanned by $5$ vector fields $X_1,X_2,Y_1,Y_2,Y_3$.  
A base of the first layer $V_1$ is given by $\{X_1,X_2\}$, and the only non-trivial bracket relations defining the Lie algebra structure of $\Lie(F_{23})$ are 
\begin{equation*}
Y_1 := [X_1,X_2], \qquad Y_2  := [X_1,Y_1] , \qquad Y_3  := [X_2,Y_1].
\end{equation*}
Choose $A:=\exp(\Span\{Y_1,Y_2,Y_3\})$. We identify $F_{23}$ with $(\mathbb{R}^5,\cdot)$, using the coordinates in Lemma \ref{l coordinates}. We choose as left-invariant orthonormal frame of the distribution the vector fields

\begin{eqnarray*}
X_1(x_1,x_2,\theta_1,\theta_2,\theta_3) & := & \frac{\partial}{\partial x_1},\\ X_2(x_1,x_2,\theta_1,\theta_2,\theta_3) & := & \frac{\partial}{\partial x_2} + x_1 \frac{\partial}{\partial \theta_1}+ \frac{x^2_1}{2}\frac{\partial}{\partial \theta_2} + x_1x_2 \frac{\partial}{\partial \theta_3};
\end{eqnarray*}
where $x_1,x_2,\theta_1,\theta_2,\theta_3\in\mathbb{R}$.
If we apply Theorem \ref{t formofHmu} to $F_{23}$, the subgroup $A$ and $\mu:=\sum_{i=1}^3 a_l\df \theta_l$, with $a_i\in\mathbb{R}$ for $i\in\{1,...,3\}$, by \eqref{e formofHmuexamples} the reduced Hamiltonian $H_\mu\in C^{\infty}(T^*\mathbb{R}^2)$ is 
\begin{equation*}
 H_{\mu}(p_x,x)  = \frac{1}{2} p_{x_1}^{2} + \frac{1}{2}\left(p_{x_2}+a_1 x_1 + a_{2} \frac{x_1^2}{2} + a_3 x_1x_2\right)^2, \forall (p_x,x)\in T^*\mathbb{R}^2.
\end{equation*}
The Hamiltonian flow of $H_\mu$ is Arnold-Liouville integrable. Indeed, for all $\mu\in\Lie(A)^*$ we can exhibit two meromorphic prime integrals that are in involution, $H_\mu$ and $C_\mu\in C^\infty(T^*\mathbb{R}^2)$, where the second is given by
$$ C_\mu := 
 a_3 p_x -a_2 p_y + a_1 a_3 x_2 + \frac{a_3^2}{2} x_2^2.$$
 Since $H_\mu$ and $C_\mu$ are smooth as functions of $\mu$, by Corollary \ref{c consequences} we have that the normal Hamiltonian flow in $F_{23}$ is Arnold-Liouville integrable.
The integrability of the normal flow in $F_{23}$ is already known in literature, we refer for example to \cite[ Exercise 7.80]{agrachev}.

\subsubsection{$N_{6,2,5a^*}$ as $2$-abelian extension}

Let $N_{6,2,5a^*}$ be the 
Carnot group 
with Lie algebra spanned by $6$ vectors $X_1,X_2,Y_1,Y_2,Y_3,Y_4$ satisfying the following bracket relations
\begin{equation*}
\begin{split}
Y_1 &= [X_1,X_2], \;\; Y_2  = [X_1,Y_1] ,\\
 \;\; Y_3  &= [X_2,Y_1], \;\; Y_4  = [X_1,Y_2] = [X_2,Y_3]. \\
  [X_2,Y_4]&=[Y_2,X_2]=[X_1,Y_4]=[X_1,Y_3]=0.
\end{split}
\end{equation*}
Using the coordinates defined in Lemma \ref{l coordinates} we identify $N_{6,2,5a^*}$ with $\mathbb{R}^6$. Choose as left-invariant orthonomal frame of the distribution the two vector fields $X_1,X_2\in\Vecc(\mathbb{R}^6)$ defined by 
\begin{eqnarray*}
    X_1(x_1,x_2,\theta) &:=& \frac{\partial}{\partial x_1} \\ X_2(x_1,x_2,\theta) &:=& \frac{\partial}{\partial x_2} + x_1 \frac{\partial}{\partial \theta_1}+ \frac{x^2_1}{2}\frac{\partial}{\partial \theta_2} + x_1x_2 \frac{\partial}{\partial \theta_3}  + (\frac{x_1^3}{3!} + \frac{x_1x_2^2}{2}) \frac{\partial}{\partial \theta_4}, 
\end{eqnarray*}
for all $x_1,x_2\in\mathbb{R}$, for all $\theta\in\mathbb{R}^4$.

Apply Theorem \ref{t formofHmu} with $G=N_{6,2,5a^*}$, $A:=\exp(\Span\{Y_i\}_{i=1,...,4})$ and $\mu:=\sum_{i=1}^4 a_i\df \theta_i$, with $a_i\in\mathbb{R}$ for all $i\in\{1,...,4\}$. By \eqref{e formofHmuexamples} we have that the reduced Hamiltonian $H_\mu\in C^{\infty}(T^*\mathbb{R}^2)$ is
\begin{equation*}
 H_{\mu}(p_x,x)  = \frac{1}{2}\left( p_{x_1}^{2} + \left(p_{x_2}+a_1 x_1 + a_{2} \frac{x_1^2}{2} + a_3 x_1x_2 + a_4 \left(\frac{x_1^3}{3!}+\frac{x_1x_2^2}{2}\right)  \right)^2\right),
\end{equation*}
for all  $(p_x,x)\in  T^*\mathbb{R}^2$.
The normal Hamiltonian flow is Arnold-Liouville integrable. Indeed, 
it was proven in \cite[Section 5]{kruglikov2017integrability} that the function
\begin{equation}
I(p_x,p_y,x,y) := P_{X_1} P_{Y_3} -  P_{X_2} P_{Y_2} + \frac{1}{2}P_{Y_1}^2,
\end{equation}
is an $A$-invariant prime integral in involution with the normal Hamiltonian and with the momentum map, where for $Z\in\Vecc(N_{6,2,5a^*})$ the function $P_{Z}\in C^\infty(T^*N_{6,2,5a^*})$ is defined by $P_Z(\lambda):=\langle\lambda ,Z\rangle$.

\subsubsection{$F_{24}$ as $2$-abelian extension}

Let $F_{24}$ be the free-nilpotent Lie algebra with rank $2$ and step $4$.
The Lie algebra of $F_{24}$ is spanned by the $8$ vectors $X_1,X_2,Y_1,...,Y_6$. The Lie algebra structure is defined by the only non trivial bracket relations 

\begin{equation*}
\begin{split}
Y_1 &:= [X_1,X_2], \qquad Y_2  := [X_1,Y_1] , \qquad Y_3  := [X_2,Y_1], \\
Y_4 &:= [X^1,Y_2], \qquad Y_5 : = [X_1,Y_3] = [X_2,Y_4], \qquad Y_6:= [X_2,Y_3]. \\
\end{split}
\end{equation*}
Identify $F_{24}$ with $(\mathbb{R}^8,\cdot)$ using the coordinates in Lemma \ref{l coordinates}.
Choose as orthonormal frame for the distribution the two vectorfields $X_1,X_2\in\Vecc(\mathbb{R}^8)$ defined by
\begin{eqnarray*}
    X_1(x_1,x_2,\theta) & := &\frac{\partial}{\partial x_1} \\
    X_2(x_1,x_2,\theta) & := &\frac{\partial}{\partial x_2} + x_1 \frac{\partial}{\partial \theta_1}+ \frac{x^2_1}{2}\frac{\partial}{\partial \theta_2} + x_1x_2 \frac{\partial}{\partial \theta_3}  + \frac{x_1^3}{3!}\frac{\partial}{\partial \theta_4} +\\  &&+\frac{x_1^2x_2}{2}\frac{\partial}{\partial \theta_5}+ \frac{x_1x_2^2}{2} \frac{\partial}{\partial \theta_6}, 
\end{eqnarray*}
for all $x_1,x_2\in\mathbb{R}$, for all $\theta\in\mathbb{R}^6$.
Apply Theorem \ref{t formofHmu} with $G=F_{24}$, with $A=\exp(\Span\{Y_i\}_{i=1,...,6})$ and $\mu=\sum_{i=1}^6 a_i\df \theta_i$, with $a_i\in\mathbb{R}$ for $i\in\{1,...,6\}$. By \eqref{e formofHmuexamples} the reduced Hamiltonian is 
\begin{equation*}
 H_{\mu}\left(p_x,x\right)  = \frac{1}{2} p_{x_1}^{2} + \frac{1}{2}\left(p_{x_2}+a_1 x_1 + a_{2} \frac{x_1^2}{2} + a_3 x_1x_2 + a_4 \frac{x_1^3}{3!}+ a_5\frac{x_1x_2^2}{2}+a_6\frac{x_2^2x_1}{2}\right)^2,
\end{equation*}
for all $(p_x,x)\in T^*\mathbb{R}^2$.

\subsection{Engel-type groups}\label{s engel_type}
Denote with $\Eng(n)$ the Carnot group with Lie algebra spanned by $2n+2$ vectors $X_1,...,X_n,Y_0,...,Y_{n+1}$ satisfying as only non trivial bracket relations 
\begin{equation}
Y_i := [X_i,Y_0], \qquad Y_{n+1} := [X_i,Y_i].
\label{e engel_type_brackets}
\end{equation}
These groups have been studied in \cite{ledonnemoisala}. Using the coordinates described in Lemma \ref{l coordinates}, we identify the group $\Eng(n)$ with $\mathbb{R}^{2n+2}$. We choose as orthonormal left-invariant frame of the distribution the vector fields $X_1,...,X_n,Y_0\in \Vecc(\mathbb{R}^{2n+2})$ defined for all $x\in\mathbb{R}^n,\theta\in \mathbb{R}^{n+2}$ by
\begin{eqnarray*}
    X_i(x,\theta) & := & \frac{\partial}{\partial x_i}, \ \ \forall i\in\{1,...,n\},\\
    Y_0(x,\theta) & := & \frac{\partial}{\partial \theta_0}+\sum_{i=1}^n x_i\frac{\partial}{\partial \theta_i}+\frac{x_i^2}{2}\frac{\partial}{\partial \theta_{n+i}}.
\end{eqnarray*}
Apply Theorem \ref{t formofHmu} with $G=\Eng(n)$, $A:=\Span(\{Y_l\}_{l\in \{ 0,...,n+1\}})$ and $\mu:=\sum_{i=0}^{n+1} a_i\df\theta_i$. By \eqref{e formofHmuexamples} we have that the reduced Hamiltonian $H_\mu\in C^\infty(\mathbb{R}^n)$ is of the form
\begin{equation}
 H_{\mu}(p_x,x)  = \frac{1}{2} ||p_x||_{\text{eu}}^{2} + \frac{1}{2} \left(a_0+\sum_{i=1}^{n} a_i x_i +  \frac{a_{n+1}}{2}||x||_{\text{eu}}^2\right)^2, \ \ \forall (p_x,x)\in T^*\mathbb{R}^n,
 \label{e reduced_Hamiltonian_Engel}
\end{equation}
where $||\cdot||_{\text{eu}}$ denotes the euclidean norm. Since for all $\mu \in \Lie(A)^*$ we get the Hamiltonian of an harmonic oscillator, the Hamiltonian flow of the reduced Hamiltonian $H_\mu$ is Arnold-Liouville integrable for all $\mu\in \Lie(A)^*$ with set of prime integrals smoothly depending on $\mu$ (see for example \cite{central_forces_integrability}).
In particular, as an immediate consequence of Corollary~\ref{c consequences}, we get Corollary \ref{c Eng}.\\
The integrability of the normal Hamiltonian flow in $\Eng(n)$ can be proven directly exhibiting $2n+2$ independent prime integrals that are in involution, we refer to Appendix \ref{appendix_eng} for an extended presentation.

\appendix
\section{Proof of Lemma \ref{l hill_region}}\label{appendix_proof_Hill_region}
\begin{proof}[Proof of Lemma \ref{l hill_region}]
We claim that we can choose $\epsilon>0$ such that for all $y\in V^{-1}([\frac{1}{2}-\epsilon,\frac{1}{2}])\cap \Omega$ we have $\nabla V(y)\neq 0$ (we denote with $\nabla$ the gradient with respect to the chosen scalar product).
Indeed, if by contradiction this was not possible, we could find a sequence of points $y_n\in V^{-1}([\frac{1}{2}-\frac{1}{n},\frac{1}{2}])\cap\Omega $, with $n\in\mathbb{N}$, for which $\nabla V(y_n)=0$. Being $\Omega$ compact, up to sub-sequence we can assume that $\{y_n\}_{n\in\mathbb{N}}$ converges to some $y\in \Omega$. By the smoothness of $V$ we would have $V(y)=\frac{1}{2}$ and $\nabla V(y)=0$, this would contradict the hypothesis that $\frac{1}{2}$ is a regular value for $V$ restricted to some neighborhood of $\Omega$.
Denote $\Omega_\epsilon:= V^{-1}([\frac{1}{2}-\epsilon,\frac{1}{2}])\cap \Omega$ and $\Omega_{\frac{\epsilon}{2}}:= V^{-1}([\frac{1}{2}-\frac{\epsilon}{2},\frac{1}{2}])\cap \Omega $.
Set $M:=\sup_{y\in\Omega, v\in B(0,1)}\sum_{i,j}\frac{\partial^2V}{\partial_{x_i}\partial_{x_j}}(y)v_iv_j $. Up to change $\epsilon$ we can assume that 
\begin{equation}
    \delta=\delta(\epsilon):=\inf_{y\in\Omega_\epsilon}||\nabla V(y)||^2-M\epsilon>0
    \label{e estimatedelta}
\end{equation}
Indeed, $\delta(\epsilon)$ is increasing as $\epsilon$ decreases to $0$, and $\delta(0)>0$.\\
We prove equation \eqref{e avaregeofxdot} in two steps:
\begin{enumerate}
    \item We show that there exists $C_\epsilon>0$ such that, for all $t_1,t_2\in\mathbb{R}$, with $t_1<t_2$, if $(q(t),y(t)):[t_1,t_2]\rr T^*\mathbb{R}^n$ solves the Hamilton equation for $H$, $H(q(t_1),y(t_1))=\frac{1}{2}$ and one of the two following condition holds
    \begin{enumerate}
        \item $y(t)\in \Omega\setminus \Omega_{\frac{\epsilon}{2}}$ for all $t\in [t_1,t_2]$;
        \item $V(y(t_1))=V(y(t_2))=\frac{1}{2}-\epsilon$, $y(t)\in \Omega_\epsilon$ for all $t\in[t_1,t_2]$, and there exists $\Bar{t}\in (t_1,t_2)$ such that $y(\Bar{t})\in\Omega_{\frac{\epsilon}{2}}$;
    \end{enumerate}
    then 
    \begin{equation*}
        \frac{1}{t_2-t_1}\int_{t_1}^{t_2}\sqrt{1-||\dot{y}(t)||^2}<1-C_\epsilon.
    \end{equation*}
    \item We prove that there exists a sequence $\{t_i\}_{i\in\mathbb{N}}$ with $0\leq t_i<t_{i+1}$ and $\lim_{i\rr \infty}t_i=+\infty$, such that for all $i>1$ one between conditions (a) and (b) of point (1) holds for $x|_{[t_i,t_{i+1}]}$.   
\end{enumerate}
If (1) and (2) hold it is trivial to prove that
\begin{equation*}
    \limsup_{T\rr\infty} \frac{1}{T}\int_0^T\sqrt{1-||\dot{x}(t)||^2}\df t\leq 1-C_\epsilon.
\end{equation*}
Therefore, proving (i) and (ii) we prove \eqref{e avaregeofxdot} and we conclude the proof of the lemma.\\
We start proving point (i): set $C:=\sup_\Omega ||\nabla V||>0$ and
\begin{equation}
    C_\epsilon:=\min\bigg\{1-\sqrt{1-\frac{\epsilon}{2}}, \frac{\delta}{2C^2}\left(1-\sqrt{1-\frac{\epsilon}{2}}\right)\bigg\}.
    \label{e defC_epsilon}
\end{equation}
Being $H(q(t),y(t))=\frac{1}{2}$ for all $t\in [t_1,t_2]$ and $\dot{y}=q$, we have
\begin{equation}
    ||\dot{y}||=\sqrt{\frac{1}{2}-V(y(t))},\quad \forall t\in[t_1,t_2].
    \label{e normydot}
\end{equation}
If condition (a) holds, then $V(y(t))\leq \frac{1}{2}-\frac{\epsilon}{2}$ for all $t\in[t_1,t_2]$. Consequently, by \eqref{e normydot} we have $||\dot{y}||\geq \sqrt{\frac{\epsilon}{2}}$ and therefore
\begin{equation*}
     \frac{1}{t_2-t_1}\int_{t_1}^{t_2}\sqrt{1-||\dot{y}(t)||^2}\df t < \sqrt{1-\frac{\epsilon}{2}}\leq 1-C_\epsilon.
\end{equation*}
Assume now that condition (b) holds. For all $t\in[t_1,t_2]$ we have
\begin{eqnarray}
(V(y(t)))'= \nabla V\cdot \dot{y}(t);
\label{e firstderivativepotential}\\
    (V(y(t)))''=-||\nabla V(y(t))||^2+\sum_{i,j}\frac{\partial^2V}{\partial_{x_i}\partial_{x_j}}(y(t))\dot{y}_i(t)\dot{y}_j(t).
    \label{e secondderivativepotential}
\end{eqnarray}
For all $t\in[t_1,t_2]$, since $y(t)\in\Omega_\epsilon$ we have by \eqref{e normydot} that 
\begin{equation}
    ||\dot{y}(t)||\leq \sqrt{\epsilon}.
    \label{e upperboundnormy}
\end{equation}  
By the definition of $\delta$ in \eqref{e estimatedelta} and equations \eqref{e firstderivativepotential}, \eqref{e secondderivativepotential} and \eqref{e upperboundnormy}, we have for all $t\in[t_1,t_2]$ that
\begin{eqnarray}
    V(y(t))'\leq C\sqrt{\epsilon}, 
    \label{e boundfirstderivativeonstrip}\\
    V(y(t))''\leq -\delta.
    \label{e boundsecondderivativeonstrip}
\end{eqnarray}
Equations \eqref{e boundfirstderivativeonstrip} and \eqref{e boundsecondderivativeonstrip} imply
\begin{eqnarray*}
    \frac{1}{2}-\epsilon=V(t_2)\leq V(y(t_1))+V'(t_1)(t_2-t_1)-\delta (t_2-t_1)^2\\
    \leq \frac{1}{2}-\epsilon+C\sqrt{\epsilon}(t_2-t_1)-\delta (t_2-t_1)^2.
\end{eqnarray*}
Therefore, 
\begin{equation}
    t_2-t_1\leq \frac{C\sqrt{\epsilon}}{\delta}.
    \label{e exittime}
\end{equation}
Since there exists $\Bar{t}$ such that $y(\Bar{t})\in\Omega_{\frac{\epsilon}{2}}$, and since $V(y(t_1))=\frac{1}{2}-\epsilon$ and $y(t)\in \Omega_\epsilon$ for all $t\in[t_1,t_2]$, there exist $t_3\in[t_1,t_2]$ such that $V(y(t_3))=\frac{1}{2}-\frac{\epsilon}{2}$ and $y(t)\in\Omega_{\epsilon}\setminus\Omega_{\frac{\epsilon}{2}}$ for all $t\in [t_1,t_3]$. By \eqref{e boundfirstderivativeonstrip} we have
\begin{equation}
    t_3-t_1\geq \frac{\sqrt{\epsilon}}{2C}.
    \label{e boundshortinterval}
\end{equation}
By \eqref{e normydot} we have $||\dot{y}(t)||\geq \sqrt{\frac{\epsilon}{2}}$ for all $t\in[t_1,t_3]$. Therefore
\begin{equation}
  \int_{t_1}^{t_3}\sqrt{1-||\dot{y}(t)||^2}\df t\leq (t_3-t_1)\sqrt{1-\frac{\epsilon}{2}}.
    \label{e estimateintegraleonshortpath}
\end{equation}
We have
\begin{eqnarray*}
\frac{1}{t_2-t_1}\int_{t_1}^{t_2}\sqrt{1-||\dot{y}(t)||^2}\df t&=&\\ 
&=&\frac{1}{t_2-t_1}\int_{t_1}^{t_3}\sqrt{1-||\dot{y}(t)||^2}\df t \\ &&+\frac{1}{t_2-t_1}\int_{t_3}^{t_2}\sqrt{1-||\dot{y}(t)||^2}\df t \\
&\stackrel{\eqref{e estimateintegraleonshortpath}}{\leq} &\frac{(t_3-t_1)}{t_2-t_1}\sqrt{1-\frac{\epsilon}{2}}+\frac{t_2-t_3}{t_2-t_1}\\
&=&1+\frac{(t_3-t_1)}{t_2-t_1}(\sqrt{1-\frac{\epsilon}{2}}-1)\\
&\stackrel{\eqref{e exittime},\eqref{e boundshortinterval}}{\leq}& 1- \frac{\delta}{2C^2}(1-\sqrt{1-\frac{\epsilon}{2}})
\stackrel{\eqref{e defC_epsilon}}{\leq} 1-C_\epsilon.
\end{eqnarray*}
thus we concluded the proof of (i).\\
We prove now (ii). Up to choose $\epsilon$ small we can assume $x(0)\notin \Omega_\epsilon$. Choose $t_1=0$. Define for all $k\in\mathbb{N}$, with $k>0$,
\begin{eqnarray*}
    t_{2k}:= &\inf\big\{t>t_{2k-1} :  &V(x(t))=\frac{1}{2}-\epsilon \text{ and }\\ & &s>t,V(x(s))=\frac{1}{2}-\epsilon \implies \exists s'\in(t,s) \text{ s.t. }x(s')\in \Omega_{\frac{\epsilon}{2}}\big\}
  \\
    t_{2k+1}:=&\inf\big\{t>t_{2k} : &V(x(t))=\frac{1}{2}-\epsilon \big\}.
\end{eqnarray*}
It is trivial to check that for all $k\in\mathbb{N}$ we have that condition (a) holds for $x|_{[t_{2k+1},t_{2k+2}]}$ and property (b) holds for $x|_{[t_{2k},t_{2k+1}]}$, thus the proof of point (ii) (and therefore the proof of the lemma) is concluded.
\end{proof}

\section{Integrability of the normal Hamiltonian flow in $\Eng(n)$}\label{appendix_eng}
This appendix is dedicated to an alternative proof of the integrability of the normal Hamiltonian flow in Engel-type groups. We refer to Section \ref{s engel_type} for the definition of $\Eng(n)$. The proof we present follows \cite{central_forces_integrability}.\\
Fix $n\in\mathbb{N}$. Choose a base $X_1,....,X_n,Y_0,...,Y_{n+1}$ of $\Lie(\Eng(n))$ satisfying \eqref{e engel_type_brackets}, with $X_1,...,X_n,Y_0$ orthonormal base of the first layer. For $X\in \Vecc(\Eng(n))$ define $P_X\in C^\infty(T^*M)$ setting
\begin{equation}
    P_X(\lambda)=\langle\lambda,X\rangle, \ \ \forall \lambda\in T^*G.
    \label{e def_P_X}
\end{equation}
We remark that the normal Hamiltonian (see \eqref{e defnormalhamiltonian}) is 
\begin{equation}
    H=\sum_{i=1}^n P_{X_i}^2+P_{Y_0}^2.
    \label{e normal_Hamiltonian_P_X}
\end{equation}
When computing Poisson brackets we will make constant use of the following relation
: for all left-invariant vector-fields $X,Y\in \Vecc(T^*M)$ we have
\begin{equation}
\{P_X,P_Y\}=P_{[X,Y]}.
\label{e poisson_bracket_P_X}
\end{equation}
For, $i,j  \in\{1,...,n\}, i\neq j$ and $N\in \{1,...,n\}$ define $L_{ij},C_N\in C^\infty(T^*\Eng(n))$
\begin{equation}
\begin{split}
 L_{ij } & :=P_{X_i}P_{Y_j} - P_{X_j} P_{Y_i};\\
 C_{N} & := \frac{1}{2} \sum_{\stackrel{i,j\in\{1,...,N\}}{i\neq j}}  L_{ij}^2.
\end{split}
\label{e def_L_ij}
\end{equation}
We start proving that the $L_{ij}$'s and the $C_N$'s are constant along the flow of the normal Hamiltonian.
\begin{lemma}\label{lem:cons-mot-En(n)}
The functions  $L_{ij}$ and $C_N$ are 
prime integrals for the normal Hamiltonian flow in $\Eng(n)$.    
\label{l appendix_B_prime_integrals}
\end{lemma}

\begin{proof}
Fix $i,j\in \{1,...,n\}$, $i\neq j$. To prove that $L_{ij}$ is constant along the flow of the normal Hamiltonian $H$ we prove that $\{L_{ij},H\}=0$:
\begin{equation*}
\begin{split}
\{ L_{ij} , H \}  \stackrel{\eqref{e def_L_ij}}{=} & P_{X_i} \{P_{Y_j}, H \} +                        P_{Y_j} \{ P_{X_i},H \}  - P_{X_j}  \{P_{Y_i},H \} -  P_{Y_i} \{ P_{X_j},H \} \\
                 \stackrel{\eqref{e normal_Hamiltonian_P_X},\eqref{e poisson_bracket_P_X}}{=} &   - P_{X_i}P_{X_j}P_{Y_{n+1}} + P_{Y_j} P_{Y_0} P_{Y_i} + P_{X_j}P_{X_i} P_{Y_{n+1}} - P_{Y_i}P_{Y_0} P_{Y_j} = 0. \\
\end{split}
\end{equation*}
$C_N$ is constant along the flow of the normal Hamiltonian since it is the sum of functions that are constant along the flow of the normal Hamiltonian.
\end{proof}
In the next three lemmas we compute some Poisson brackets that we will need to prove that the set of prime integrals we consider in the proof of Theorem \eqref{the1:eng-n} are in involution.
\begin{lemma}
Fix $i,j,k,l\in\{1,...,n\}$, with $i\neq j$ and $k\neq l$. We have 
\begin{equation}
\begin{split}
    \{ L_{ij},L_{kl} \} &= P_{Y_{n+1}} (\delta_{ik}L_{jl} +\delta_{jl} L_{ik} - \delta_{il} L_{jk} - \delta_{jk}L_{il}).\\
\end{split}
\label{e bracket_L_L}
\end{equation}
\end{lemma}

\begin{proof}
 We start the proof computing the Poisson brackets $\{P_{X_i},L_{kl}\}$ and $\{P_{Y_j},L_{kl}\}$: we have
\begin{equation}
\begin{split}
  \{ P_{X_i},L_{kl}\}  \stackrel{\eqref{e def_L_ij}}{=}& P_{X_k}\{P_{X_i},P_{Y_l} \} + P_{Y_l}\{P_{X_i},P_{X_k} \} \\ & - P_{X_l} \{ P_{X_i}, P_{Y_k} \} - P_{Y_k} \{ P_{X_i}, P_{X_l} \} \\
                     \stackrel{\eqref{e poisson_bracket_P_X}}{=}& P_{Y_{n+1}} (P_{X_k}\delta_{il} - P_{X_l} \delta_{ik}),
\end{split}
\label{e bracket_L_X}
\end{equation}
and
\begin{equation}
\begin{split}
    \{ P_{Y_j},L_{kl} \}  \stackrel{\eqref{e def_L_ij}}{=} & P_{X_k}\{P_{Y_j},P_{Y_l} \} + P_{Y_l}\{P_{Y_j},P_{X_k} \}\\ &- P_{X_l} \{ P_{Y_j}, P_{Y_k} \} - P_{Y_k} \{ P_{Y_j}, P_{X_l} \} \\
                     \stackrel{\eqref{e poisson_bracket_P_X}}{=} & P_{Y_{n+1}} (-P_{Y_l}\delta_{jk} + P_{Y_k} \delta_{jl}).  
\end{split}
\label{e bracket_L_Y}
\end{equation}
Using the two equations above we compute $\{ L_{ij},L_{kl} \}$ proving the lemma:
\begin{equation*}
\begin{split}
  \{ L_{ij},L_{kl} \} \stackrel{\eqref{e def_L_ij}}{=} &  P_{X_i}\{P_{Y_j},L_{kl} \} + P_{Y_j}\{P_{X_i},L_{kl} \}\\ & - P_{X_j} \{ P_{Y_i}, L_{kl} \} - P_{Y_i} \{ P_{X_j}, L_{kl} \} \\
                    \stackrel{\eqref{e bracket_L_X},\eqref{e bracket_L_Y}}{=} &  P_{Y_{n+1}} (P_{X_i} (-P_{Y_l}\delta_{jk} + P_{Y_k} \delta_{jl}) + P_{Y_j} (P_{X_k}\delta_{il} - P_{X_l} \delta_{ik})   \\
                    & - P_{X_j} (-P_{Y_l}\delta_{ik} + P_{Y_k} \delta_{il}) - P_{Y_i} (P_{X_k}\delta_{jl} - P_{X_l} \delta_{jk}) ) \\
                  & =   P_{Y_{n+1}} (\delta_{ik}L_{jl} +\delta_{jl} L_{ik} - \delta_{il} L_{jk} - \delta_{jk}L_{il}) .\\
\end{split}
\end{equation*}
\end{proof}

\begin{lemma}\label{lem:cons-inv-En(n)}
The functions $L_{ij}$ and $C_N$ satisfy 
\begin{equation}
\begin{split}
    \{ C_N, L_{kl}\} &  = 0 \,\,\,\,  \text{if} \;\;  N \leq k < l \;\; \text{or} \;\; k < l \leq N. \\
\end{split}
\label{e C_n_Lk}
\end{equation}
\end{lemma}

\begin{proof}
For $N\in\{1,...,n\}$ and $k,l\in\{1,...,n\}$ we have 
\begin{equation}
\begin{split}
    \{ C_N, L_{kl}\} &  \stackrel{\eqref{e def_L_ij}}{=} \sum_{i,j = 1 }^N L_{ij} \{L_{ij},L_{kl} \} \\ &\stackrel{\eqref{e bracket_L_L}}{=} P_{Y_{n+1}} \sum_{i,j = 1 }^N L_{ij} (\delta_{ik}L_{jl} +\delta_{jl} L_{ik} - \delta_{il} L_{jk} - \delta_{jk}L_{il}).
\end{split}    
\label{e bracket_C_L_1}
\end{equation}
We prove \eqref{e C_n_Lk} for $k < l \leq N$, the case $N \leq k < l$ is analogous.
  If $k < l < N $, then   $\delta_{ik}$, $\delta_{jl}$, $\delta_{il} L_{jk}$ and $\delta_{jk}$ are zero and the lemma is trivially true. Assume now $k < l = N$. We have 
  \begin{equation*}
\begin{split}
    \{ C_N, L_{kl}\} &  \stackrel{\eqref{e bracket_C_L_1}}{=} P_{Y_{n+1}} \sum_{i < j \leq N } L_{ij} (\delta_{ik}L_{jl} +\delta_{jl} L_{ik} - \delta_{il} L_{jk} - \delta_{jk}L_{il}) \\
    & + P_{Y_{n+1}} \sum_{j < i\leq N  } L_{ij} (\delta_{ik}L_{jl} +\delta_{jl} L_{ik} - \delta_{il} L_{jk} - \delta_{jk}L_{il})=0 \\
\end{split}    
\end{equation*}
where in the last equality we used $L_{ij}=-L_{ji}$.
\end{proof}

We can finally prove the integrability of the normal Hamiltonian flow in $\Eng(n)$.
\begin{Theorem}\label{the1:eng-n}
The normal Hamiltonian flow in $\Eng(n)$ is Arnold-Liouville integrable.
\end{Theorem}

\begin{proof}
We consider first the case of $n$ even. Write $n = 2v$ for some $v\in\mathbb{N}$. We claim that $\{ H, L_{1,2}, L_{3,4}, \dots , L_{2v-1,2v}, C_4, C_6,  \cdots , C_{2v}\}$ is a set of independent prime integrals for the flow of the normal Hamiltonian $H$ that are in involution. Indeed, by Lemmas \ref{l appendix_B_prime_integrals}-\ref{lem:cons-inv-En(n)} the above functions are all prime integrals that are in involution. The independence of these functions easily follows by their definitions (see \eqref{e def_L_ij}).
When $n$ is odd, we write $n = 2v+1$ for some $v\in\mathbb{N}$ we can prove in a similar way that $\{H, L_{2,3}, L_{4,5}, \dots , L_{2v,2v+1}, C_2, C_4,  \cdots , C_{2v} \}$ is a set of independent prime integrals for the flow of the normal Hamiltonian that are in involution. 
\end{proof}

\bibliographystyle{plain}
\bibliography{bibliography.bib}

\end{document}